\documentclass[a4paper,12pt,reqno]{amsart}
\usepackage{amsfonts, color}
\usepackage{amsmath}
\usepackage{amssymb}
\usepackage[a4paper]{geometry}
\usepackage{mathrsfs}
\usepackage[colorlinks]{hyperref}
\usepackage{hyperref}
\renewcommand\eqref[1]{(\ref{#1})}

\makeatletter
\newcommand*{\mint}[1]{%
  \mint@l{#1}{}%
}
\newcommand*{\mint@l}[2]{%
  \@ifnextchar\limits{%
    \mint@l{#1}%
  }{%
    \@ifnextchar\nolimits{%
      \mint@l{#1}%
    }{%
      \@ifnextchar\displaylimits{%
        \mint@l{#1}%
      }{%
        \mint@s{#2}{#1}%
      }%
    }%
  }%
}
\newcommand*{\mint@s}[2]{%
  \@ifnextchar_{%
    \mint@sub{#1}{#2}%
  }{%
    \@ifnextchar^{%
      \mint@sup{#1}{#2}%
    }{%
      \mint@{#1}{#2}{}{}%
    }%
  }%
}
\def\mint@sub#1#2_#3{%
  \@ifnextchar^{%
    \mint@sub@sup{#1}{#2}{#3}%
  }{%
    \mint@{#1}{#2}{#3}{}%
  }%
}
\def\mint@sup#1#2^#3{%
  \@ifnextchar_{%
    \mint@sup@sub{#1}{#2}{#3}%
  }{%
    \mint@{#1}{#2}{}{#3}%
  }%
}
\def\mint@sub@sup#1#2#3^#4{%
  \mint@{#1}{#2}{#3}{#4}%
}
\def\mint@sup@sub#1#2#3_#4{%
  \mint@{#1}{#2}{#4}{#3}%
}
\newcommand*{\mint@}[4]{%
  \mathop{}%
  \mkern-\thinmuskip
  \mathchoice{%
    \mint@@{#1}{#2}{#3}{#4}%
        \displaystyle\textstyle\scriptstyle
  }{%
    \mint@@{#1}{#2}{#3}{#4}%
        \textstyle\scriptstyle\scriptstyle
  }{%
    \mint@@{#1}{#2}{#3}{#4}%
        \scriptstyle\scriptscriptstyle\scriptscriptstyle
  }{%
    \mint@@{#1}{#2}{#3}{#4}%
        \scriptscriptstyle\scriptscriptstyle\scriptscriptstyle
  }%
  \mkern-\thinmuskip
  \int#1%
  \ifx\\#3\\\else_{#3}\fi
  \ifx\\#4\\\else^{#4}\fi
}
\newcommand*{\mint@@}[7]{%
  \begingroup
    \sbox0{$#5\int\m@th$}%
    \sbox2{$#5\int_{}\m@th$}%
    \dimen2=\wd0 %
    \let\mint@limits=#1\relax
    \ifx\mint@limits\relax
      \sbox4{$#5\int_{\kern1sp}^{\kern1sp}\m@th$}%
      \ifdim\wd4>\wd2 %
        \let\mint@limits=\nolimits
      \else
        \let\mint@limits=\limits
      \fi
    \fi
    \ifx\mint@limits\displaylimits
      \ifx#5\displaystyle
        \let\mint@limits=\limits
      \fi
    \fi
    \ifx\mint@limits\limits
      \sbox0{$#7#3\m@th$}%
      \sbox2{$#7#4\m@th$}%
      \ifdim\wd0>\dimen2 %
        \dimen2=\wd0 %
      \fi
      \ifdim\wd2>\dimen2 %
        \dimen2=\wd2 %
      \fi
    \fi
    \rlap{%
      $#5%
        \vcenter{%
          \hbox to\dimen2{%
            \hss
            $#6{#2}\m@th$%
            \hss
          }%
        }%
      $%
    }%
  \endgroup
}
\numberwithin{equation}{section}
\theoremstyle{plain}
\newtheorem{thm}{Theorem}[section]
\newtheorem{prop}[thm]{Proposition}
\newtheorem{cor}[thm]{Corollary}
\newtheorem{lem}[thm]{Lemma}

\theoremstyle{definition}
\newtheorem{defn}[thm]{Definition}
\newtheorem{rem}[thm]{Remark}

\newcommand{\G}{\mathbb{G}}
\newcommand{\Ra}{\mathcal{R}^{\frac{a}{\nu}}}
\newcommand{\R}{\mathcal{R}}
\newcommand{\X}{\mathbb{X}}

\title[Logarithmic Sobolev-type inequalities on Lie groups]{Logarithmic Sobolev-type inequalities on Lie groups}
\author[M. Chatzakou]{Marianna Chatzakou}
\address{
	Marianna Chatzakou:
	\endgraf
    Department of Mathematics: Analysis, Logic and Discrete Mathematics
    \endgraf
    Ghent University, Belgium
  	\endgraf
	{\it E-mail address} {\rm marianna.chatzakou@ugent.be}
		}
	
\author[A. Kassymov]{Aidyn Kassymov}
\address{
  Aidyn Kassymov:
  \endgraf
  Institute of Mathematics and Mathematical Modeling
  \endgraf
  28 Shevchenko str.
  \endgraf
  050010 Almaty
  \endgraf
  Kazakhstan
  \endgraf
  and
   \endgraf
  Department of Mathematics: Analysis, Logic and Discrete Mathematics
  \endgraf
  Ghent University, Belgium
  \endgraf
	{\it E-mail address} {\rm kassymov@math.kz}}

\author[M. Ruzhansky]{Michael Ruzhansky}
\address{
  Michael Ruzhansky:
  \endgraf
  Department of Mathematics: Analysis, Logic and Discrete Mathematics
  \endgraf
  Ghent University, Belgium
  \endgraf
 and
  \endgraf
  School of Mathematical Sciences
  \endgraf
  Queen Mary University of London
  \endgraf
  United Kingdom
  \endgraf
  {\it E-mail address} {\rm michael.ruzhansky@ugent.be}
  }

\begin{document}

\thanks{The authors are supported by the FWO Odysseus 1 grant G.0H94.18N: Analysis and Partial Differential Equations and by the Methusalem programme of the Ghent University Special Research Fund (BOF) (Grant number 01M01021). Marianna Chatzakou is a postdoctoral fellow of the Research Foundation – Flanders (FWO) under the postdoctoral grant No 12B1223N.  Michael Ruzhansky is also supported by EPSRC grant EP/R003025/2. Aidyn Kassymov is also supported by the Science Committee of the Ministry of Science and Higher Education of the Republic of Kazakhstan (Grant No. AP23484106). }

\begin{abstract} 
In this paper we show a number of logarithmic inequalities on several classes of Lie groups: log-Sobolev inequalities on general Lie groups, log-Sobolev (weighted and unweighted),  log-Gagliardo-Nirenberg and log-Caffarelli-Kohn-Nirenberg inequalities on graded Lie groups. Furthermore, on stratified groups, we show that one of the obtained inequalities is equivalent to a Gross-type log-Sobolev inequality with the horizontal gradient. As a result, we obtain the Gross log-Sobolev inequality on general stratified groups but, {\bf very interestingly}, with the Gaussian measure on the first stratum of the group. Moreover, our methods also yield weighted versions of the Gross log-Sobolev inequality. In particular, we also obtain new weighted Gross-type log-Sobolev inequalities on $\mathbb R^n$ for arbitrary choices of homogeneous quasi-norms. As another consequence we derive the Nash inequalities on graded groups and an example application to the decay rate for the heat equations for sub-Laplacians on stratified groups. We also obtain weighted versions of log-Sobolev and Nash inequalities for general Lie groups.
\end{abstract}

\maketitle
  
\tableofcontents

\section{Introduction}

In this paper we derive the logarithmic versions of several well-known functional inequalities.  Some inequalities are obtained with best constants, or with semi-explicit constants, which are useful for some further applications. Our techniques allow us to derive these inequalities in rather general settings, so we will be working in the settings of general Lie groups, as well as on several classes of nilpotent Lie groups, namely, graded and homogeneous Lie groups. Since on stratified Lie groups we also have the horizontal gradient at our disposal, we will also formulate versions of some of the inequalities in the setting of stratified groups, using the horizontal gradient instead of a power of a sub-Laplacian. 

In the Euclidean spaces, in one of the Sobolev's pioneering works, Sobolev obtained the following inequality, which at this moment is bearing his name:
\begin{equation}
    \|u\|_{L^{p^{*}}(\mathbb{R}^{n})}\leq C \|\nabla u\|_{L^{p}(\mathbb{R}^{n})},
\end{equation}
where $1<p<n$, $p^{*}=\frac{np}{n-p}$ and $C=C(n,p)>0$ is a positive constant. The best constant of this inequality was obtained by Talenti in \cite{T76} and Aubin in \cite{A76}. The Sobolev inequality is one of the most important tools in studying PDE and variational problems.
Folland and Stein extended Sobolev's inequality to general stratified  groups (see e.g. \cite{GV}):
if $\mathbb{G}$ is a stratified group and $\Omega\subset \mathbb{G}$ is an open set, then there exists a constant $C>0$ such that  we have
\begin{equation}\label{Sobolev-Folland-Stein}
\|u\|_{L^{p^{*}}(\Omega)}\leq C\left(\int_{\Omega}|\nabla_{H} u|^{p}dx\right)^{\frac{1}{p}},\; 1<p<Q,\; p^{*}=\frac{Qp}{Q-p},
\end{equation}
for all $u\in C_{0}^{\infty}(\Omega)$. Here $\nabla_{H}$ is the horizontal gradient and $Q$ is the homogeneous dimension of $\mathbb{G}$.
Inequality \eqref{Sobolev-Folland-Stein} is called the Sobolev or Sobolev-Folland-Stein inequality. 
Furthermore, in relation to groups, we can mention
Sobolev inequalities and embeddings on general unimodular Lie groups \cite{VSCC93}, on general locally compact unimodular groups \cite{AR20}, on general noncompact Lie groups \cite{BPTV19,BPV21,Gro92}, as well as Hardy-Sobolev inequalities on general Lie groups \cite{RY18a}. 

Particularly, in the seminal paper ``Logarithmic Sobolev inequalities on Lie groups'' \cite{Gro92} L. Gross obtained the logarithmic Sobolev inequality on general Lie groups with respect to the heat kernel measure. Here we give inequalities with respect to the Haar measure, from which we also derive Gross type inequalities but with the semi-Gaussian measures on the first stratum, in the case when a group is stratified. In fact, we also show an equivalence of some of such inequalities.

The Sobolev inequality on graded groups using Rockland operators was proved in \cite{FR17} and the best constant for it was obtained in \cite{RTY20}. The compactness of such embeddings was also studied in \cite{GKR22}.

Furthermore, the logarithmic Sobolev inequality was shown to hold on $\mathbb{R}^n$ in the following form: 
\begin{equation}\label{ddlogin}
\int_{\mathbb{R}^{n}}\frac{|u|^{p}}{\|u\|^{p}_{L^{p}(\mathbb{R}^{n})}}\log\left(\frac{|u|^{p}}{\|u\|^{p}_{L^{p}(\mathbb{R}^{n})}}\right)dx\leq\frac{n}{p}\log\left(C\frac{\|\nabla u\|^{p}_{L^{p}(\mathbb{R}^{n})}}{\|u\|^{p}_{L^{p}(\mathbb{R}^{n})}}\right).
\end{equation}
We can refer to \cite{Wei78} for the case $p=2$, but to e.g. \cite{DD03} for some history review of cases 
$1< p<n$, including the discussion of best constants.

In \cite{Mer08}, the author obtained a logarithmic Gagliardo-Nirenberg inequality. In \cite {FNQ18} and \cite{KRS20}  the authors proved the logarithmic Sobolev inequality and the fractional logarithmic Sobolev inequality on the Heisenberg group and on homogeneous  groups, respectively. A fractional weighted version of \eqref{ddlogin} on homogeneous groups was proved in \cite{KS20}. In this paper, we prove logarithmic Sobolev inequalities on graded groups and weighted logarithmic Sobolev inequalities on general Lie groups. As applications of these inequalities we show Nash and weighted Nash inequalities on graded and general Lie groups, respectively. The log-Sobolev type inequalities with some specific orders of weights are also sometimes called the log-Hardy inequalities \cite{DDFT10}.

The subject of the logarithmic inequalities has been extensively investigated and it is impossible to give a reasonably complete  review of the literature here. We refer to surveys \cite{AB00, GZ03}, to works on relations to other inequalities \cite{BL00, MF93}, coercive inequalities on Carnot groups \cite{Bou21, BZ21a, BZ21b, BZ21c, HZ09}, Nash \cite{Nas58, Bec98, BDS20, OS18} and weighted Nash \cite{BBGM12} inequalities. There are works on the fractional Laplacian \cite{Bec12} as well as on the 
logarithmic Sobolev inequalities for the fractional Laplacian  \cite{KRS20} in the setting of Folland and Stein (\cite{FS}) homogeneous groups.
Logarithmic inequalities have various applications in different areas, see e.g. \cite{BH99, Car91, FNQ18, Mer08, OW16}.
Explicit constants for some inequalities play an important role (e.g. \cite{T76, Bec75, DD03}), for 
best constants on graded Lie groups see \cite{RTY20}. We refer to Davies \cite{Dav89} for a general exposition linking log-Sobolev and Nash inequalities to the hypercontractivity of the heat semigroup.

In view of extending Sobolev inequalities to infinite dimensional spaces, in \cite{Gro75}, L. Gross showed the logarithmic Sobolev inequality in 
the following form: 
\begin{equation}\label{EQ:Gross}
    \int_{\mathbb{R}^{n}}|u(x)|^{2}\log  \left(\frac{|u(x)|}{\|u\|_{L^{2}(\mu)}}\right)d\mu(x)
    \leq  \int_{\mathbb{R}^{n}}|\nabla u(x)|^{2}d\mu(x),
\end{equation}
with the Gaussian measure $d\mu(x)=(2\pi)^{-n/2}e^{-|x|^2/2}dx$.
The importance of \eqref{EQ:Gross} and the use of the Gaussian measure is that it has a uniform constant independent of the dimension, allowing passing to infinite dimensional spaces, and in particular also leading to the ultracontractivity and hypercontractivity properties of the corresponding Markovian semigroups. Due to its importance in many applications, this subject has been intensively investigated, see e.g.
\cite{Ad79, AC79, Ros76, Wei78, Tos97, SZ92}, to mention only very few papers. It is related to uncertainty principles \cite{Bec95} and to
Poincar\'e inequalities \cite{Bec89}, with the 
Poincar\'e and logarithmic Sobolev inequalities having different forms \cite{Gen14} and applications, see e.g. \cite{AS94}. We can also refer to the recent work on the
Poincar\'e inequality on filiform type stratified groups \cite{CFZ21}.

In this paper, among other things, we prove the logarithmic Gross inequality on stratified groups, with the gradient $\nabla u$ in \eqref{EQ:Gross} replaced by the horizontal gradient $\nabla_H u$, collecting vector fields from the first stratum. Moreover, {\bf very curiously}, our Gross inequality \eqref{gaus.log.sobi} comes with the Gaussian measure on the first stratum. This appears to be a natural extension of the Euclidean case since such inequality turns out to be equivalent to the $L^2$-log-Sobolev inequality with the horizontal gradient, see Theorem \ref{equiv.thm}.

An interesting result in this perspective is \cite[Theorem 6.3]{HZ09} where Hebisch and Zegarlinski have shown that on stratified groups for probability measures $\mu=e^{-\beta d^p}/Z\, d\lambda$, $\beta\geq 0$, $p\geq 1$, corresponding to smooth homogeneous norms $d(x)$, there is no $L^q$-log-Sobolev inequalities of the type \eqref{EQ:Gross}, for $q\in (1,2]$. Furthermore, they went on to use the heat kernel on the Heisenberg group $\mathbb{H}^n$ to construct a probability measure $\mu$ (with a singular norm) on $\mathbb{H}^n$ such that the $L^2$-log-Sobolev inequality of the Gross type \eqref{gaus.log.sobi} holds on $\mathbb{H}^n$. Our result in Theorem \ref{semi-g} gives such measure on general stratified groups with the probability component on the first stratum, while Theorem \ref{equiv.thm} shows that one can not do better if one wants to start with the inequality
\eqref{cor.log.sobi}, which is clearly true independently of any further choice of measure.

There is an extensive literature on conditions for log-Sobolev inequalities to hold, such as Bakry-Emery and other criteria. Such analysis is not part of this work where we do not work with full probability measures.

As already mentioned, the log-Sobolev inequalities have many applications. For example, in \cite{KRS20} the authors applied the obtained log-Sobolev inequalities for the fractional Laplacian on homogeneous groups, to show blow-up results for some classes of nonlinear PDEs. Applications of the current results are still possible and will be addressed in the sequel.

The methods of this paper allow us to also derive new results for the Euclidean space $\mathbb R^n$, yielding weighted Gross type log-Sobolev inequalities for weights given in terms of arbitrary homogeneous quasi-norms. To give an example, let us consider the family of Euclidean norms on $\mathbb R^n$ given by
$|x|_p=(|x_1|^p+\cdots+|x_n|^p)^{1/p}$ for $1\leq p<\infty$, and by $|x|_\infty=\max\limits_{1\leq j\leq n}|x_j|$, for $x\in\mathbb R^n.$ 
Let $0\leq \beta <2<n$. 
Let $\mu_p$ be the Gaussian measure on $\mathbb{R}^{n}$ given by 
$d\mu_p=\gamma_p e^{-\frac{|x|_2^2}{2}}dx$, $1\leq p\leq\infty$, with the normalisation constant $\gamma_p$ given in \eqref{k-w-Rn}, depending on $p$.
Then we have the following weighted Gross-type log-Sobolev inequality:
\begin{equation}
    \label{gaus.log.sob-w-Rn-i}
    \int_{\mathbb R^n}|x|_p^{-\frac{\beta(n-2)}{n-\beta}}|g|^2\log\left(|x|_p^{-\frac{\beta(n-2)}{2(n-\beta)}}|g|\right)\,d\mu_p \leq \int_{\mathbb R^n} |\nabla g|^2\,d\mu_p,\quad 
    0\leq \beta <2<n,
\end{equation}
for all $g$ such that $\||x|_p^{-\frac{\beta(n-2)}{2(n-\beta)}}g\|_{L^2(\mu_p)}=1.$
We note that for $\beta=0$ this implies the classical Gross inequality \eqref{EQ:Gross}, and we refer to Corollary \ref{semi-g-w-Rn} for further details. 
The case of $\beta=2$ in \eqref{gaus.log.sob-w-Rn-i} corresponds to the logarithmic Hardy inequality and requires very different methods from those developed in this paper; this case will be addressed in \cite{CKR21}. For a discussion of weights in inequalities of the type \eqref{ddlogin}, but only with one weight, we can refer to \cite{Das21}.

The case $p=1$ of the log-Sobolev inequalities in this paper corresponds to the Shannon type inequalities; we will not discuss it here but refer to \cite{CKR21a} for this case in the setting of general homogeneous groups, where it was handled using rather different methods than those in this paper.

To guide the reader through different inequalities in this paper, let us briefly summarise the obtained results. We formulate some of them in a simplified form here, with the full statements given in the sequel of this paper.

\begin{itemize}
\item {\bf Logarithmic inequalities on general Lie groups.}
We will establish several versions of log-Sobolev inequalities on general Lie groups. The only (non-technical) assumption that we impose is that the group is connected, {\em otherwise it can be compact or noncompact, unimodular or non-unimodular, and can have polynomial or exponential volume growth at infinity. }
We refer to Section \ref{SEC:generalLie} for the exact setup and statements, but let us give an example. Namely, let $\G$ be a connected Lie group with a H\"ormander system of vector fields, of local dimension $d$. Then for any $1<p<\infty$ and $0<a<\frac{d}{p}$ we have the log-Sobolev inequality
     \begin{equation}\label{sobinnon1-i}
         \int_{\mathbb{G}}\frac{|u|^{p}}{\|u\|^{p}_{L^{p}(\lambda)}}\log\left(\frac{|u|^{p}}{\|u\|^{p}_{L^{p}(\lambda)}}\right) d\lambda(x) \leq  \frac{d}{ap}\log\left(C\frac{\|u\|^{p}_{L^{p}_{a}(\lambda)}}{\|u\|^{p}_{L^{p}(\lambda)}}\right),
     \end{equation}
where $\lambda$ is the left Haar measure on $\G$, and $L^p_a$
is the Sobolev space associated to the corresponding sub-Laplacian (with drift). In fact, \eqref{sobinnon1-i} is a special case (without weight) of a family of weighted log-Sobolev inequalities given in Theorem \ref{loghardysob}. To give an example, if $q>p>1$, $0\leq \beta < d$ and $0<a<d$  are such that $\frac{1}{p}-\frac{1}{q}=\frac{a}{d}-\frac{\beta}{dq}$, then we have
     \begin{multline*}\label{harsobinnon1-i}
         \int_{\mathbb{G}}\frac{|x|_{CC}^{-\frac{\beta p}{q}}|u|^{p}}{\||x|_{CC}^{-\frac{\beta}{q}} u\|^{p}_{L^{p}(\lambda)}}\log\left(\frac{|x|_{CC}^{-\frac{\beta p}{q}}|u|^{p}}{\||x|_{CC}^{-\frac{\beta}{q}} u\|^{p}_{L^{p}(\lambda)}}\right) d\lambda(x)  \leq \frac{d-\beta}{ap-\beta}\log\left(C\frac{\|u\|^{p}_{L^{p}_{a}(\lambda)}}{\||x|_{CC}^{-\frac{\beta}{q}} u\|^{p}_{L^{p}(\lambda)}}\right),
     \end{multline*}
where  $|x|_{CC}=d_{CC}(e,x)$ is the Carnot-Carath\'{e}odory distance between $x$ and the identity element $e$ of $\G$.  As one can see, \eqref{sobinnon1-i} is a special case of this inequality with $\beta=0$.
The inequality \eqref{sobinnon1-i} for $a\geq \frac{d}{p}$ will be also given in Theorem \ref{logsobcor}. Finally we note that this last estimate also holds true with the Carnot-Carath\'eodory distance replaced by the Riemannian distance, see Remark \ref{REM:CCR}.
    \item {\bf Logarithmic inequalities on graded groups.}
    We obtain several versions of log-Sobolev (weighted and unweighted), log-Gagliardo-Nirenberg and log-Caffarelli-Kohn-Nirenberg inequalities on graded groups, let us give one example. Let $\mathbb{G}$ be a graded group with homogeneous dimension $Q$, and let $\R$ be a positive Rockland operator of homogeneous degree $\nu$. Let $1<p<\infty$ and $0<a<\frac{Q}{p}.$ Then we have
\begin{equation}\label{a2=0i}
\int_{\mathbb{G}}\frac{|u|^{p}}{\|u\|^{p}_{L^{p}(\mathbb{G})}}\log\left(\frac{|u|^{p}}{\|u\|^{p}_{L^{p}(\mathbb{G})}}\right)dx\leq \frac{Q}{ap}\log\left(A\frac{\|u\|^p_{\dot{L}^p_a(\G)}}{\|u\|^{p}_{L^{p}(\G)}}\right),
\end{equation}
with $\|u\|_{\dot{L}^p_a(\G)}\equiv
\|\R^\frac{a}{\nu}u\|_{{L}^{p}(\G)}$,
for a constant $A$ that can be related to the best constants in the Gagliardo-Nirenberg inequalities, see \eqref{CinfGN}.
This inequality \eqref{a2=0i} can be seen as a refinement of 
\eqref{sobinnon1-i}, not only with respect to a better control over a constant under the logarithm, but also because one has the homogeneous Sobolev norm on the right hand side of \eqref{a2=0i}.
\item {\bf Log-Sobolev and Gross inequalities on stratified groups.}
Let $\G$ be a stratified Lie group with homogeneous dimension $Q$. As a special case of the obtained inequalities on graded groups, we have the  log-Sobolev inequality 
\begin{equation}
    \label{cor.log.sobi}
    \int_{\mathbb{G}}|u|^2 \log|u|\,dx \leq \frac{Q}{4}\log \left(A \int_{\mathbb{G}}|\nabla_{H}u|^2\,dx \right)\,,
\end{equation}
for every $u$ such that $\|u\|_{L^2(\G)}=1$, 
and where $\nabla_H$ is the horizontal gradient, and the constant $A$ is related to the best constants in the Gagliardo-Nirenberg inequalities, see 
\eqref{Cinf}.
Furthermore, let $n_1$ be the dimension of the first stratum of the Lie algebra of $\G$, and let us write $x \in \G$ as $x=(x',x'') \in \mathbb{R}^{n_1} \times \mathbb{R}^{n-n_1}$, where $n$ is the topological dimension of $\G$. Then the following Gross-type ``semi-Gaussian'' log-Sobolev inequality is satisfied
\begin{equation}
    \label{gaus.log.sobi}
    \int_{\G}|g|^2\log|g|\,d\mu \leq \int_{\G} |\nabla_{H}g|^2\,d\mu\,,
\end{equation}
for any $g$ such that $\|g\|_{L^2(\mu)}=1,$
 where $\mu=\mu_1 \otimes \mu_2$, and $\mu_1$ is the Gaussian measure on $\mathbb{R}^{n_1}$ given by $d\mu_1=\gamma e^{-\frac{|x'|^2}{2}}dx'$, for $x'\in \mathbb{R}^{n_1}$ and $|x'|$ being its Euclidean norm, where
$\gamma:=(4^{-1}Q e^{\frac{2n_1}{Q}{-1}}A)^{Q/2},$
 and $\mu_2$ is the Lebesgue measure $dx''$ on $\mathbb{R}^{n-n_1}$.
While both of \eqref{cor.log.sobi} and \eqref{gaus.log.sobi} hold true, they are also {\bf equivalent} to each other, in the sense that one inequality can be shown to imply the other, see Theorem \ref{equiv.thm}. This equivalence is well-known in the Euclidean space, and here it can be also viewed as a justification for considering the Gaussian measure on the first stratum of $\G$. Since in the Euclidean space the first stratum is the whole space, these results recover the well-known Euclidean results as well, with
\eqref{gaus.log.sobi} boiling down to
\eqref{EQ:Gross} in the case of $\mathbb{R}^n$, see the discussion before Theorem \ref{semi-g}.

\item {\bf Weighted Gross log-Sobolev inequalities.} 
The weighted log-Sobolev inequalities in Theorem \ref{wsobingrthm} give the following weighted analogue of \eqref{cor.log.sobi}: if
$\G$ is a stratified Lie group with homogeneous dimension $Q$, $|\cdot|$ is an arbitrary homogeneous quasi-norm on $\G$, and $0\leq \beta <2<Q$, then there exists $C>0$ such that the following weighted log-Sobolev inequality holds:
  \begin{equation}\label{harsobingr-2-i}
         \int_{\mathbb{G}}{|x|^{-\frac{\beta(Q-2)}{Q-\beta}}|u|^{2}}\log\left({|x|^{-\frac{\beta(Q-2)}{2(Q-\beta)}}|u|}\right) dx \leq \frac{Q-\beta}{2(2-\beta)}\log\left(C\int_\G |\nabla_H u|^{2} dx\right),
     \end{equation}
for every $u$ such that $\||x|^{-\frac{\beta(Q-2)}{2(Q-\beta)}}u\|_{L^2(\G)}=1$.
Consequently, \eqref{harsobingr-2-i} implies the following weighted version of the Gross log-Sobolev inequality:
Let $\G$ be a stratified group with homogeneous dimension $Q$ and let $|\cdot|$ be an arbitrary homogeneous quasi-norm on $\G$. Then for any $0\leq \beta <2<Q$ we have the following weighted ``semi-Gaussian'' log-Sobolev inequality
\begin{equation}
    \label{gaus.log.sob-w-i}
    \int_{\G}|x|^{-\frac{\beta(Q-2)}{Q-\beta}}|g|^2\log\left(|x|^{-\frac{\beta(Q-2)}{2(Q-\beta)}}|g|\right)\,d\mu \leq \int_{\G} |\nabla_{H}g|^2\,d\mu\,,
\end{equation}
for all $g$ such that $\||x|^{-\frac{\beta(Q-2)}{2(Q-\beta)}}g\|_{L^2(\mu)}=1,$
 where $\mu=\mu_1 \otimes \mu_2$, with $\mu_1$ the Gaussian measure on the first stratum of $\G$ given by $d\mu_1=\gamma e^{-\frac{|x'|^2}{2}}dx'$ with $|x'|$ being the Euclidean norm of $x'$, 
 and $\mu_2$ is the Lebesgue measure $dx''$ on the other strata. In particular, we get \eqref{gaus.log.sobi} for $\beta=0$. We refer to 
 Theorem \ref{semi-g-w} for more explanations and for the precise value of $\gamma.$ If $\G=\mathbb R^n$ is the Euclidean space, this yields a new family of weighted Gross type log-Sobolev inequalities \eqref{gaus.log.sob-w-Rn-i} on $\mathbb R^n.$

\item {\bf Nash inequality.}
It is well-known in Euclidean spaces that the $L^2$-log-Sobolev inequality is equivalent to the Nash inequality. We show that this remains true also in more generality. 
Let $\G$ be a graded group with homogeneous dimension $Q$ and let $\mathcal{R}$ be a positive Rockland operator of homogeneous degree $\nu$. Then, we show an extended Nash inequality
 \begin{equation}\label{Nash-final2i}
 \|u\|_{L^2(\G)}^{2+\frac{4a}{Q}} \leq A_2 \|u\|_{L^1(\G)}^{\frac{4a}{Q}}\|u\|^{2}_{\dot{L}^2_a(\G)}\,.
 \end{equation}
 or, equivalently, 
\begin{equation}
    \label{Nash.gi}
    \left[\int_{\G} |u|^2\,dx \right]^{{1+\frac{2a}{Q}}} \leq A_2 \left[ \int_{\G}|u|\,dx\right]^{\frac{4a}{Q}}
    \int_{\G} |\mathcal{R}^{\frac{a}{\nu}}u|^2\,dx \,,
\end{equation}
for all $0<a<\frac{Q}{2}$, with constant $A_{2}$ related to the best constants in the Gagliardo-Nirenberg inequalities, see Theorem \ref{THM:Nash}.
In particular, if $\G$ is a stratified Lie group, then for $a=1$ this implies 
\begin{equation}
    \label{Nasi}
    \|u\|_{L^2(\G)}^{2+\frac{4}{Q}} \leq A_{2}
    \|u\|_{L^1(\G)}^{\frac{4}{Q}}
    \|\nabla_H u\|_{L^2(\G)}^2,
\end{equation}
with the horizontal gradient $\nabla_H.$ For the Euclidean $\G=\mathbb R^n$  this is the full gradient, yielding as a special case of $a=1$ the original Nash inequality in \cite{Nas58} as well as, again for $a=1$, the Nash inequality on stratified groups in \cite{VSCC93}.

As a simple immediate application, in Corollary \ref{cor:par}
we show that if $\Delta_\G$ is a sub-Laplacian on a stratified group $\G$ with $Q\geq 3$, and $u_0\in L^1(\G)\cap L^2(\G)$ is non-negative $u_0\geq 0$, then the solution $u$  to the heat equation 
\begin{equation}
    \label{heat.eqi}
    \partial_t u=\Delta_\G u,\quad u(0,x)=u_0(x),
\end{equation}
satisfies the time-decay estimate
\begin{equation}
    \label{heatesti}
    \|u(t,\cdot)\|_{L^2(\G)}\leq \left( \|u_0\|_{L^2(\G)}^{-\frac{4}{Q}}
    +\frac{4}{QA_{2}}\|u_0\|_{L^1(\G)}^{-\frac{4}{Q}}t\right)^{-\frac{Q}{4}},
\end{equation}
for all $t\geq 0,$ where $A_{2}$ is a constant in \eqref{Nash.gi}.

Finally, in Theorem \ref{THM:Nash-gen} we show Nash's inequality as well as its weighted versions on general connected Lie groups. Namely, if $d$ is a local dimension of $\G$, and $0\leq \beta <2a<d$, 
then the weighted Nash inequality is given by 
 \begin{equation}\label{Nash-final2th-i}
\||\cdot|_{CC}^{-\frac{\beta(d-2a)}{2(d-\beta)}}u\|_{L^2(\lambda)}^{2+\frac{2(2a-\beta)}{d-\beta}} \leq C \||\cdot|_{CC}^{-\frac{\beta(d-2a)}{2(d-\beta)}}u\|_{L^1(\lambda)}^{\frac{2(2a-\beta)}{d-\beta}}\|u\|^{2}_{L^2_a(\lambda)},
 \end{equation}
 where $\lambda$ is the left Haar measure on $\G$ and $|x|_{CC}=d_{CC}(e,x)$ is the Carnot-Carath\'{e}odory distance between $x$ and $e$ (but it can be also replaced by the Riemannian distance, see Remark \ref{REM:CCR}).

 In particular, for $\beta=0$, we obtain the Nash inequality on general connected Lie groups in the following form:
  \begin{equation}\label{Nash-final2th0-i}
 \|u\|_{L^2(\lambda)}^{2+\frac{4a}{d}} \leq C \|u\|_{L^1(\lambda)}^{\frac{4a}{d}}\|u\|^{2}_{L^2_a(\lambda)},\quad 0<a<\frac{d}{2}.
 \end{equation}
 For $a\geq\frac{d}{2}$, we also get the weighted and unweighted versions of the Nash inequality, in particular, for any $q>2$, we have the Nash type inequality
  \begin{equation}\label{Nash-final2th0-i2}
     \|u\|^{2+\frac{2(q-2)}{q}}_{L^{2}(\lambda)}\leq C\|u\|^{\frac{2(q-2)}{q}}_{L^{1}(\lambda)}\|u\|^{2}_{L^{2}_{a}(\lambda)}, \quad a\geq\frac{d}{2}.
 \end{equation}

\end{itemize}

\section{Preliminaries}
In  t=his  section,  we  briefly  recall  definitions and  main  properties  of  the homogeneous, graded and stratified groups.
The comprehensive analysis on such groups has been initiated in the works of Folland and Stein \cite{FS}, but in our exposition below we follow a more recent presentation in the open access book \cite{FR16}. 

\subsection{Homogeneous groups} In this sub-section, we give definitions and main properties of homogeneous Lie group.
\begin{defn}[\cite{FS, FR16}, Homogeneous group]
A Lie group (on $\mathbb{R}^{N}$) $\mathbb{G}$ with the dilation
$$D_{\lambda}(x):=(\lambda^{\nu_{1}}x_{1},\ldots,\lambda^{\nu_{N}}x_{N}),\; \nu_{1},\ldots, \nu_{n}>0,\; D_{\lambda}:\mathbb{R}^{N}\rightarrow\mathbb{R}^{N},$$
which is an automorphism of the group $\mathbb{G}$ for each $\lambda>0,$
is called a {\em homogeneous (Lie) group}.
\end{defn}
For simplicity,  in this paper we use the notation $\lambda x$ for the dilation $D_{\lambda}(x)$.
We
denote 
\begin{equation}
Q:=\nu_{1}+\ldots+\nu_{N},
\end{equation}
the homogeneous dimension of a homogeneous group $\mathbb{G}$. 
Let $dx$ denote the Haar measure on $\mathbb{G}$ and let $|S|$ denote the corresponding volume of a measurable set $S\subset \mathbb{G}$.
Then we have
\begin{equation}\label{scal}
|D_{\lambda}(S)|=\lambda^{Q}|S| \quad {\rm and}\quad \int_{\mathbb{G}}f(\lambda x)
dx=\lambda^{-Q}\int_{\mathbb{G}}f(x)dx.
\end{equation}
We also note that from \cite[Proposition 1.6.6]{FR16}, the standard Lebesgue measure $dx$ on $\mathbb{R}^{N}$ is the Haar measure on $\G$. 
\begin{defn}[{\cite[Definition 3.1.33]{FR16} or \cite[Definition 1.2.1]{RS19}}]\label{quasi-norm}
For any homogeneous group $\mathbb{G}$ there exist  homogeneous quasi-norms, which are  continuous non-negative functions
\begin{equation}
\mathbb{G}\ni x\mapsto |x|\in[0,\infty),
\end{equation}
with the properties

\begin{itemize}
\item[a)] $|x|=|x^{-1}|$ for all $x\in\mathbb{G}$,
\item[b)] $|\lambda x|=\lambda|x|$ for all $x\in \mathbb{G}$ and $\lambda>0$,
\item[c)] $|x|=0$ if and only if $x=0$.
\end{itemize}
\end{defn}
They are quasi-norms since the triangle inequality may be satisfied with a constant different from $1.$
\subsection{Graded Lie groups}\label{SEC:graded}
In this subsection, we make a brief summary of the basic definitions and properties of the graded Lie groups.  
\begin{defn}[e.g. {\cite[Definition 3.1.1]{FR16}}, graded Lie group and graded Lie algebra]
A Lie algebra $\mathfrak{g}$ is
called {\em graded} if it is endowed with a vector space decomposition (where all but
finitely many of the $V_{j}$'s are $0$)
\begin{equation}
\mathfrak{g}=\oplus_{j=1}^{\infty}V_{j},\,\,\,\,\text{s.t.}\,\,\,[V_{i},V_{j}]\subset V_{i+j}.
\end{equation}
Consequently, a Lie group is called {\em graded} if it is a connected and simply connected
Lie group whose Lie algebra is graded.
\end{defn}
Before introducing the Rockland operator, we recall the  Rockland condition.
By $\pi$ and $\widehat{\G}$ we denote a representation and the unitary dual of $\G$, respectively, and by $\mathcal{H}_{\pi}^{\infty}$ we denote the smooth vectors of the representation $\pi\in\widehat{\G}$. Let us recall the Rockland condition: 
\begin{defn}[Rockland condition, see e.g. {\cite[Definition 4.1.1]{FR16}}]
Let $A$ be a left-invariant differential operator on a Lie group $\G$.
Then $A$ satisfies the {\em Rockland condition} when

(Rockland condition) for each representation $\pi\in\widehat{\G}$, except for the trivial representation,
the operator $\pi(A)$ is injective on $\mathcal{H}_{\pi}^{\infty}$, that is,
\begin{equation}
\forall v \in\mathcal{H}_{\pi}^{\infty},\,\,\, \pi(A)v =0 \Rightarrow v = 0.
\end{equation}
\end{defn}
\begin{defn}[Rockland operator, see e.g. {\cite[Definition 4.1.2]{FR16}}]\label{intRockland}
Let $\G$ be a homogeneous Lie group. A {\em Rockland operator} $\R$ on
$\G$ is a left-invariant differential operator which is homogeneous of positive degree
and satisfies the Rockland condition.
\end{defn}
Let us recall a property which  connects   homogeneous Lie groups and Rockland operators.
\begin{prop}[e.g. {\cite[Proposition 4.1.3]{FR16}}]\label{exishomgr}
Let $\G$ be a homogeneous Lie group. If there exists a Rockland operator on $\G$ then $\G$ is a graded.
\end{prop}
Similarly to the \cite{RTY20}, in this paper we will not be using the representation theoretic interpretation of these operators, so we define Rockland operators as \textit{left-invariant homogeneous hypoelliptic differential operators on $\G$.}

Let $\mathcal{R}$ be a Rockland operator of homogeneous degree $\nu$. Let us recall the definitions of the homogeneous and inhomogeneous Sobolev spaces on graded Lie groups, respectively, endowed with norms:
\begin{equation}
    \|f\|_{\dot{L}^p_{a}(\G)}:=\|\R^{\frac{a}{\nu}}f\|_{L^{p}(\G)},
\end{equation}
and
\begin{equation}
    \|f\|_{L^p_{a}(\G)}:=(\|f\|^{p}_{L^{p}(\G)}+\|\R^{\frac{a}{\nu}}f\|^{p}_{L^{p}(\G)})^{\frac{1}{p}}.
\end{equation}
 
It was shown in \cite{FR16} that while these norms depend on the choice of $\mathcal{R}$, the Sobolev spaces do not. In this paper, we  use notation  $L^{p}_{a_{1},a_{2}}(\G)=\dot{L}^{p}_{a_{1}}(\G)\cap \dot{L}^{p}_{a_{2}}(\G)$. 
\subsection{Stratified Lie groups}\label{defstr}
In this section, we recall the  definition of stratified groups (homogeneous Carnot groups) and their basic properties. Let us briefly recall the definition of the stratified Lie group. We refer to e.g. \cite{BLU07}, \cite{FR16} and \cite{RS19} for further discussions in this direction. 
\begin{defn} \label{maindef}
A Lie group $\mathbb{G}=(\mathbb{R}^{n},\circ)$ is called a stratified Lie group if it satisfies the following assumptions:

(a) For some natural numbers $n_{1}+...+n_{r}=n$
the decomposition $\mathbb{R}^{n}=\mathbb{R}^{n_{1}}\times...\times\mathbb{R}^{n_{r}}$ is valid, and
for every $\lambda>0$ the dilation $\delta_{\lambda}: \mathbb{R}^{n}\rightarrow \mathbb{R}^{n}$
given by
$$\delta_{\lambda}(x)\equiv\delta_{\lambda}(x^{(1)},...,x^{(r)}):=(\lambda x^{(1)},...,\lambda^{r}x^{(r)})$$
is an automorphism of the group $\mathbb{G}.$ Here $x^{(k)}\in \mathbb{R}^{n_{k}}$ for $k=1,...,r.$

(b) Let $n_{1}$ be as in (a) and let $X_{1},...,X_{n_{1}}$ be the left invariant vector fields on $\mathbb{G}$ such that
$X_{k}(0)=\frac{\partial}{\partial x_{k}}|_{0}$ for $k=1,...,n_{1}.$ Then
$${\rm rank}({\rm Lie}\{X_{1},...,X_{n_{1}}\})=n,$$
for every $x\in\mathbb{R}^{n},$ i.e. the iterated commutators
of $X_{1},...,X_{n_{1}}$ span the Lie algebra of $\mathbb{G}.$
\end{defn}

A graded Lie algebra $\mathfrak{g}$ is  stratified if $V_{1}$ generates $\mathfrak{g}$ as an algebra. In this case, if $\mathfrak{g}$ is
nilpotent of step $m$ we have
\begin{equation}
\mathfrak{g}=\oplus_{j=1}^{\infty}V_{j},\,\,\,\,\text{s.t.}\,\,\,[V_{i},V_{1}]\subset V_{i+1},
\end{equation}
and the natural dilations $\mathfrak{g}$ are given by
\begin{equation}
    D_{r}\left(\sum_{k=1}^{m}X_{k}\right)=\sum_{k=1}^{m}r^{k}X_{k},\,\,\,\,(X_{k}\in V_{k}).
\end{equation}
Consequently, a Lie group is  stratified if it is connected and simply-connected
Lie group whose Lie algebra is stratified.

\section{Logarithmic H\"older inequalities}

The aim of this section is to show logarithmic H\"{o}lder's inequality on general  measure spaces which will be instrumental in our proof of other logarithmic inequalities. For completeness, we first record a useful auxiliary consequence of the H\"older inequality.

\begin{lem}\label{weightedholder}
Let $\X$ be a  measure space. Suppose that $1<p\leq r\leq q\leq\infty$, $a\in[0,1]$,  $u\in L^{p}(\mathbb{X})\cap L^{q}(\mathbb{X})$ with
\begin{equation}
\frac{1}{r}=\frac{a}{p}+\frac{1-a}{q}.
\end{equation}
Then we have
\begin{equation}\label{intin}
\|u\|_{L^{r}(\mathbb{X})}\leq \|u\|^{a}_{L^{p}(\mathbb{X})}\|u\|^{1-a}_{L^{q}(\mathbb{X})}.
\end{equation}
\end{lem}
\begin{proof}
From H\"{o}lder's inequality we get
\begin{equation*}
\begin{split}
\|u\|^{r}_{L^{r}(\mathbb{X})}&=\int_{\mathbb{X}}|u(x)|^{r}dx=\int_{\mathbb{X}}(|u(x)|)^{ar}(|u(x)|)^{(1-a)r}dx
\\&\leq\left(\int_{\mathbb{X}}|u(x)|^{p}dx\right)^{\frac{ar}{p}}\left(\int_{\mathbb{X}}|u(x)|^{q}dx\right)^{\frac{(1-a)r}{q}}
\\&=\|u\|^{ar}_{L^{p}(\mathbb{X})}\|u\|^{(1-a)r}_{L^{q}(\mathbb{X})},
\end{split}
\end{equation*}
with
$
\frac{ar}{p}+\frac{(1-a)r}{q}=1,
$
implying \eqref{intin}.
\end{proof}
The inequality in Lemma \ref{weightedholder} is sometimes called Littlewood's inequality, see \cite[Theorem 5.5.1 (ii), p. 55]{Gar07}. We thank Y. Sarantopoulos for drawing our attention to this reference.

Now let us show what we can call a logarithmic H\"{o}lder's inequality on general measure spaces. A weighted version of this inequality on homogeneous groups was obtained in \cite[Lemma 3.2]{KRS20}. We now give its general version on measure spaces.
\begin{lem}[Logarithmic H\"older inequality]\label{holder}
Let $\X$ be a  measure space. Let $u\in L^{p}(\mathbb{X})\cap L^{q}(\mathbb{X})\setminus\{0\}$ with some $1<p<q< \infty.$ 
Then we have
\begin{equation}\label{holdernn}
\int_{\mathbb{X}}\frac{|u|^{p}}{\|u\|^{p}_{L^{p}(\mathbb{X})}}\log\left(\frac{|u|^{p}}{\|u\|^{p}_{L^{p}(\mathbb{X})}}\right)dx\leq \frac{q}{q-p}\log\left(\frac{\|u\|^{p}_{L^{q}(\mathbb{X})}}{\|u\|^{p}_{L^{p}(\mathbb{X})}}\right).
\end{equation}
\end{lem}
\begin{proof}
We will prove inequality  \eqref{holdernn} for $u$ being a simple function. Then, using the density of the simple functions in the $L^p(\mathbb{X}) \cap L^{q}(\mathbb{X})$ space, we get the desired result, since the constant in \eqref{holdernn} depends only on $p$ and $q$. To begin with, let us consider the following function for $u$ being a simple function on $\mathbb{X}$
\begin{equation}\label{1/r}
F\left(\frac{1}{r}\right)=\log\left(\|u\|_{L^{r}(\mathbb{X})}\right). 
\end{equation}
By using Lemma \ref{weightedholder}, we have
\begin{equation}\label{convex11}
 F\left(\frac{1}{r}\right)\leq aF\left(\frac{1}{p}\right)+(1-a)F\left(\frac{1}{q}\right),
\end{equation}
with $a\in[0,1]$ and $\frac{1}{r}=\frac{a}{p}+\frac{1-a}{q}$.
Also, from \eqref{1/r} we have
\begin{equation}\label{convex}
F(r)=r\log\int_{\mathbb{X}}|u(x)|^{\frac{1}{r}}dx,
\end{equation}
and by taking the derivative of \eqref{convex}, we get
\begin{multline}
F'(r)=\log\int_{\mathbb{X}}|u(x)|^{\frac{1}{r}}dx+r\left(\log\int_{\mathbb{X}}|u(x)|^{\frac{1}{r}}dx\right)'_{r}\\
=\log\int_{\mathbb{X}}|u(x)|^{\frac{1}{r}}dx-\frac{1}{r}\frac{\int_{\mathbb{X}}|u(x)|^{\frac{1}{r}}\log |u(x)|dx}{\int_{\mathbb{X}}|u(x)|^{\frac{1}{r}}dx}.
\end{multline}
We know from \eqref{convex11} that $F(r)$ is convex,  hence, we have
\begin{equation}
F'(r)\geq\frac{F(r')-F(r)}{r'-r},\,\,\,\,r>r'>0.
\end{equation}
and by taking $r=\frac{1}{p}$ and $r'=\frac{1}{q}$, we get
\begin{equation}\label{convex1}
p\frac{\int_{\mathbb{X}}|u|^{p}\log |u|dx}{\int_{\mathbb{X}}|u(x)|^{p}dx}-\log\int_{\mathbb{X}}|u(x)|^{p}dx\leq \frac{qp}{q-p}\log\left(\frac{\|u\|_{L^{q}(\mathbb{X})}}{\|u\|_{L^{p}(\mathbb{X})}}\right).
\end{equation}
Finally, from this last fact, we establish logarithmic H\"{o}lder's inequality
\begin{equation*}
    \begin{split}
p\frac{\int_{\mathbb{X}}|u|^{p}\log |u|dx}{\int_{\mathbb{X}}|u(x)|^{p}dx}&-\log\int_{\mathbb{X}}|u(x)|^{p}dx=p\frac{\int_{\mathbb{X}}|u|^{p}\log |u|dx}{\int_{\mathbb{X}}|u(x)|^{p}dx}-\frac{\int_{\mathbb{X}}|u(x)|^{p}\log\| u\|^{p}_{L^{p}(\mathbb{G})}dx}{\int_{\mathbb{X}}|u(x)|^{p}dx}\\&
=\frac{\int_{\mathbb{X}}|u|^{p}\log |u|^{p}dx}{\int_{\mathbb{X}}|u(x)|^{p}dx}-\frac{\int_{\mathbb{X}}|u(x)|^{p}\log\| u\|^{p}_{L^{p}(\mathbb{X})}dx}{\int_{\mathbb{X}}|u(x)|^{p}dx}\\&
=\int_{\mathbb{X}}\frac{|u|^{p}}{\|u\|^{p}_{L^{p}(\mathbb{X})}}\log\left(\frac{|u|^{p}}{\|u\|^{p}_{L^{p}(\mathbb{X})}}\right) dx\\&
\leq \frac{q}{q-p}\log\left(\frac{\|u\|^{p}_{L^{q}(\mathbb{X})}}{\|u\|^{p}_{L^{p}(\mathbb{X})}}\right),
    \end{split}
\end{equation*}
implying \eqref{holdernn}.
\end{proof}

\section{Logarithmic inequalities on graded groups}
\label{SEC:3}
In this section we show analogues of   logarithmic  inequalities on the graded groups. First we recall a few facts that will be of importance later.


Let us consider the following Schr\"{o}dinger equation,
\begin{equation}\label{problem1}
    \R_{1}^{\frac{a_{1}}{\nu_{1}}}(|\R_{1}^{\frac{a_{1}}{\nu_{1}}}u|^{p-2}\R_{1}^{\frac{a_{1}}{\nu_{1}}}u)+\R_{2}^{\frac{a_{2}}{\nu_{2}}}(|\R_{2}^{\frac{a_{2}}{\nu_{2}}}u|^{p-2}\R_{2}^{\frac{a_{2}}{\nu_{2}}}u)=|u|^{q-2}u,\,\,u\in L^{p}_{a_{1},a_{2}}(\G),
\end{equation}
where $a_{1}>a_{2}\geq 0$, $1<p<\frac{Q}{a_{1}}$, $\frac{Qp}{Q-a_{2}p}<q<\frac{Qp}{Q-a_{1}p}$, and $L^{p}_{a_{1},a_{2}}(\G)=\dot{L}^{p}_{a_{1}}(\G)\cap \dot{L}^{p}_{a_{2}}(\G).$
Let us denote the energy  functional $I:L^{p}_{a_{1},a_{2}}(\G)\rightarrow \mathbb{R}$ and the Nehari functional $J:L^{p}_{a_{1},a_{2}}(\G)\rightarrow \mathbb{R}$  acting on $L^{p}_{a_{1},a_{2}}(\G)\cap L^{q}(\G)$ in the following forms:
\begin{equation}
    I(u):=\frac{1}{p}\int_{\G}|\R_{1}^{\frac{a_{1}}{\nu_{1}}}u(x)|^{p}dx+\frac{1}{p}\int_{\G}|\R_{2}^{\frac{a_{2}}{\nu_{2}}}u(x)|^{p}dx-\frac{1}{q}\int_{\G}|u(x)|^{q}dx,
\end{equation}
and 
\begin{equation}
    J(u):=\int_{\G}|\R_{1}^{\frac{a_{1}}{\nu_{1}}}u(x)|^{p}dx+\int_{\G}|\R_{2}^{\frac{a_{2}}{\nu_{2}}}u(x)|^{p}dx-\int_{\G}|u(x)|^{q}dx.
\end{equation}
Let us introduce the Nehari manifold 
\begin{equation}
    \mathcal{N}:=\{u\in L^{p}_{a_{1},a_{2}}(\G)\setminus \{0\}:J(u)=0\},
\end{equation}
and we put
\begin{equation}\label{d}
    d_0:=\inf \{I(u): u\in \mathcal{N}\}.
\end{equation}
The above notions are related by the following theorem.
\begin{thm}[\cite{RTY20}, Theorem 4.3] \label{THM:RTYb}
Let $a_1>a_2\geq0$, $1<p<\frac{Q}{a_{1}}$ and $\frac{pQ}{Q-a_{2}p} <q< \frac{pQ}{Q-a_{1}p}$. Then, the 
Schr\"{o}dinger type equation \eqref{problem1} has a least energy solution $\varphi \in L^{p}_{a_1 ,a_2} (\G)$. Moreover, we have $d_0 = I(\varphi)$.
\end{thm}
The Sobolev embeddings on graded Lie groups have been obtained in \cite{FR17}, but here we recall its version with the best constant, which is related to the quantities described in Theorem \ref{THM:RTYb}.

\begin{thm}[\cite{RTY20},  Sobolev inequality]\label{sob} Let $\mathbb{G}$ be a graded group with homogeneous dimension $Q$ and let $\R$ be a positive Rockland operator of homogeneous degree $\nu$.
Let $a>0$, $p\in\left(1,\frac{Q}{a}\right)$ and $p<q<\frac{Qp}{Q-ap}$. Then we have 
\begin{equation}\label{sobin1}
    \left(\int_{\G}|u(x)|^{q}dx\right)^{\frac{p}{q}} \leq C\int_{\G}\left(|\Ra u(x)|^{p}+|u(x)|^{p}\right)dx,
\end{equation}
for any $u\in L^{p}_{a}(\G)$.

Moreover, let $\phi$ be a least energy solution of \eqref{problem1} (with $a=a_{1}$, $\mathcal{R}_{1}=\mathcal{R}$ and $a_{2}=0$) and let $C_{S,\mathcal{R}}$ be the smallest positive constant $C$ in \eqref{sobin1}. Then we have 
  \begin{equation}\label{bestconsob}
  C_{S,\R}=\left(\frac{apq}{apq-Q(q-p)}\int_{\G}|\phi(x)|^{p}dx\right)^{\frac{p-q}{q}}=\left(\frac{pq}{q-p}d_0\right)^{\frac{p-q}{q}},
  \end{equation}
  where $d_0$ is defined in \eqref{d}.
\end{thm}
First we present the inhomogeneous logarithmic Sobolev inequality on graded groups $\mathbb{G},$ where the Sobolev norm in the estimate is inhomogeneous.  The homogeneous norm version will be given in Corollary \ref{COR:Lpspb123} as a consequence of the Gagliardo-Nirenberg inequalities.


\begin{thm}
 Let $\mathbb{G}$ be a graded group with homogeneous dimension $Q$.
Let $1<p<\infty$ and $0<a<\frac{Q}{p}$. Then we have 
\begin{equation}\label{LogSobolev}
\int_{\mathbb{G}}\frac{|u|^{p}}{\|u\|^{p}_{L^{p}(\mathbb{G})}}\log\left(\frac{|u|^{p}}{\|u\|^{p}_{L^{p}(\mathbb{G})}}\right) dx \leq \frac{Q}{ap}\log\left(C_{p}\frac{\|u\|^{p}_{L^{p}_{a}(\G)}}{\|u\|^{p}_{L^{p}(\mathbb{G})}}\right)
\end{equation}
with the constant $C_{p}=\left( \frac{a}{Qd_0}\right)^{\frac{ap}{Q}}$ for $d_0$ as in \eqref{d}, for any non-trivial $u\in L^{p}_{a}(\G)$. 
\end{thm}
\begin{proof}
From logarithmic H\"older inequality \eqref{holdernn} we have
\begin{equation}
\int_{\mathbb{G}}\frac{|u|^{p}}{\|u\|^{p}_{L^{p}(\mathbb{G})}}\log\left(\frac{|u|^{p}}{\|u\|^{p}_{L^{p}(\mathbb{G})}}\right)dx\leq \frac{q}{q-p}\log\left(\frac{\|u\|^{p}_{L^{q}(\mathbb{G})}}{\|u\|^{p}_{L^{p}(\mathbb{G})}}\right).
\end{equation}
By the assumption we have $1\leq p<q=p^{*}=\frac{pQ}{Q-ap}$, that is,
\begin{eqnarray}\label{log.sob.}
\int_{\mathbb{G}}\frac{|u|^{p}}{\|u\|^{p}_{L^{p}(\mathbb{G})}}\log\left(\frac{|u|^{p}}{\|u\|^{p}_{L^{p}(\mathbb{G})}}\right) dx & \leq & \frac{q}{q-p}\log\left(\frac{\|u\|^{p}_{L^{q}(\mathbb{G})}}{\|u\|^{p}_{L^{p}(\mathbb{G})}}\right)\nonumber\\
& \leq  & \frac{q}{q-p} \log\left(C_{S,\R}\frac{\|u\|^{p}_{L^{p}_{a}(\G)}}{\|u\|^{p}_{L^{p}(\mathbb{G})}}\right)\,,
\end{eqnarray}
where for the last inequality we have applied \eqref{sobin1}.
Notice that the parameter $q\in \left(p, \frac{Qp}{Q-ap} \right)$ does not appear at any norm involved in \eqref{log.sob.}, one can minimise the quantity $\frac{q}{q-p}$ over $q$ and get 
\[
\int_{\mathbb{G}}\frac{|u|^{p}}{\|u\|^{p}_{L^{p}(\mathbb{G})}}\log\left(\frac{|u|^{p}}{\|u\|^{p}_{L^{p}(\mathbb{G})}}\right) dx \leq \frac{Q}{ap}\log\left(C_{p}\frac{\|u\|^{p}_{L^{p}_{a}(\G)}}{\|u\|^{p}_{L^{p}(\mathbb{G})}}\right)\,,
\]
where the minimum has been achieved in the limit case where $q \rightarrow \frac{Qp}{Q-ap}$, and the constant $C_{p}$ is now given by 
\[
C_{p}=\limsup_{q \rightarrow \frac{Qp}{Q-ap}} C_{S,\R}=\left( \frac{a}{Qd_0}\right)^{\frac{ap}{Q}}\,,
\]
where the constant $C_{S,\R}$ is given by \eqref{bestconsob}.
\end{proof}

Let us now recall the Gagliardo-Nirenberg inequality on graded groups.

\begin{thm}[\cite{RTY20}, Gagliardo-Nirenberg inequality]\label{GN} Let $\mathbb{G}$ be a graded group with homogeneous dimension $Q$ and let $\R_{1}$ and $\R_{2}$ be positive Rockland operators of homogeneous degrees $\nu_{1}$ and $\nu_{2}$, respectively.
Let $a_{1}>a_{2}\geq 0$, $p\in\left(1,\frac{Q}{a_{1}}\right)$ and $\frac{Qp}{Q-a_{2}p}<q<\frac{Qp}{Q-a_{1}p}$. Then we have 

\begin{equation}\label{gnin1}
    \int_{\G}|u(x)|^{q}dx \leq C\left(\int_{\G}|\R_{1}^{\frac{a_{1}}{\nu_{1}}} u(x)|^{p}dx\right)^{\frac{Q(q-p)-a_{2}pq}{(a_{1}-a_{2})p^{2}}}\left(\int_{\G}|\R_{2}^{\frac{a_{2}}{\nu_{2}}} u(x)|^{p}dx\right)^{\frac{a_{1}pq-Q(q-p)}{(a_{1}-a_{2})p^{2}}},
\end{equation}
for any $u\in \dot{L}^{p}_{a_{1},a_{2}}(\G)$.
Moreover, let $\phi$ be a least energy solution of \eqref{problem1}  and let $C_{GN,\mathcal{R}_{1},\R_{2},a_{1},a_{2},p,q}$ be the  best constant $C$ in \eqref{sobin1}. Then we have 
  \begin{equation}\label{bestcongn}
  \begin{split}
      C_{GN,\mathcal{R}_{1},\R_{2},a_{1},a_{2},p,q}&=\frac{(a_{1}-a_{2})pq}{a_{1}pq-Q(q-p)}\left(\frac{Q(q-p)-a_{2}pq}{a_{1}pq-Q(q-p)}\right)^{\frac{a_{2}pq-Q(q-p)}{(a_{1}-a_{2})p^{2}}}\\&
      \times \left(\frac{a_{1}pq-Q(q-p)}{(a_{1}-a_{2})(q-p)}d_0\right)^{\frac{p-q}{p}},
  \end{split}
  \end{equation}
  where $d_0$ is defined in \eqref{d}.
\end{thm}
We now have the following logarithmic  Gagliardo-Nirenberg inequality.
\begin{thm}\label{THM:GNlog}
Let $\mathbb{G}$ be a graded group with homogeneous dimension $Q$ and let $\R_{1}$ and $\R_{2}$ be positive Rockland operators of homogeneous degrees $\nu_{1}$ and $\nu_{2}$, respectively.
Let $a_{1}>a_{2}\geq 0$, $p\in\left(1,\frac{Q}{a_{1}}\right)$ and $\frac{Qp}{Q-a_{2}p}<q<\frac{Qp}{Q-a_{1}p}$. If $a_{2}\neq 0$, then for any $u\not=0$ we have 
\begin{multline}\label{LogGN1}
\int_{\mathbb{G}}\frac{|u|^{p}}{\|u\|^{p}_{L^{p}(\mathbb{G})}}\log\left(\frac{|u|^{p}}{\|u\|^{p}_{L^{p}(\mathbb{G})}}\right)dx \\ \leq \frac{q}{q-p}\log\left(C^{\frac{p}{q}}_{GN,\mathcal{R}_{1},\R_{2},a_{1},a_{2},p,q}\frac{\|u\|^{\frac{Q(q-p)-a_{2}pq}{(a_{1}-a_{2})q}}_{\dot{L}^{p}_{a_{1}}(\G)}\|u\|^{\frac{a_{1}pq-Q(q-p)}{(a_{1}-a_{2})q}}_{\dot{L}^{p}_{a_{2}}(\G)}}{\|u\|^{p}_{L^{p}(\mathbb{G})}}\right).
\end{multline}
\end{thm}
\begin{proof}
Assume that $a_{2}\neq0$. From assumption $1<p<\frac{Q}{a_{1}}$, we have $1<p<\frac{Q}{a_{1}}<\frac{Q}{a_{2}}$. It means $p<\frac{Qp}{Q-a_{2}p}<q<\frac{Qp}{Q-a_{1}p}$, then $p<q$.  Using the  Gagliardo-Nirenberg inequality \eqref{gnin1} and the logarithmic H\"{o}lder inequality \eqref{holdernn},  we have
\begin{equation*}
\begin{split}
&\int_{\mathbb{G}}\frac{|u|^{p}}{\|u\|^{p}_{L^{p}(\mathbb{G})}}\log\left(\frac{|u|^{p}}{\|u\|^{p}_{L^{p}(\mathbb{G})}}\right)dx\leq \frac{q}{q-p}\log\left(\frac{\|u\|^{p}_{L^{q}(\mathbb{G})}}{\|u\|^{p}_{L^{p}(\G)}}\right)\\&
\leq\frac{q}{q-p}\log\left(C^{\frac{p}{q}}_{GN,\mathcal{R}_{1},\R_{2},a_{1},a_{2},p,q}\frac{\left(\int_{\G}|\R_{1}^{\frac{a_{1}}{\nu_{1}}} u(x)|^{p}dx\right)^{\frac{Q(q-p)-a_{2}pq}{(a_{1}-a_{2})pq}}\left(\int_{\G}|\R_{2}^{\frac{a_{2}}{\nu_{2}}} u(x)|^{p}dx\right)^{\frac{a_{1}pq-Q(q-p)}{(a_{1}-a_{2})pq}}}{\|u\|^{p}_{L^{p}(\mathbb{G})}}\right)\\&
=\frac{q}{q-p}\log\left(C^{\frac{p}{q}}_{GN,\mathcal{R}_{1},\R_{2},a_{1},a_{2},p,q}\frac{\|u\|^{\frac{Q(q-p)-a_{2}pq}{(a_{1}-a_{2})q}}_{\dot{L}^{p}_{a_{1}}(\G)}\|u\|^{\frac{a_{1}pq-Q(q-p)}{(a_{1}-a_{2})q}}_{\dot{L}^{p}_{a_{2}}(\G)}}{\|u\|^{p}_{L^{p}(\mathbb{G})}}\right).
\end{split}
\end{equation*}
The proof is complete.
\end{proof}

If $\R$ is a positive Rockland operators of homogeneous degrees $\nu$, then, as usual, we identify
$\|u\|_{\dot{L}^p_a(\G)}\equiv
\|\R^\frac{a}{\nu}u\|_{{L}^{p}(\G)}$.

\begin{cor}\label{COR:Lpspb123}
Let $\mathbb{G}$ be a graded group with homogeneous dimension $Q$ and let $\R$ be a positive Rockland operator of homogeneous degree $\nu$.
Let $1<p<\infty$ and $0<a<\frac{Q}{p}.$
Then for any $u\not=0$ we have
\begin{equation}\label{a2=0}
\int_{\mathbb{G}}\frac{|u|^{p}}{\|u\|^{p}_{L^{p}(\mathbb{G})}}\log\left(\frac{|u|^{p}}{\|u\|^{p}_{L^{p}(\mathbb{G})}}\right)dx\leq \frac{Q}{ap}\log\left(A\frac{\|u\|^p_{\dot{L}^p_a(\G)}}{\|u\|^{p}_{L^{p}(\G)}}\right),
\end{equation}
with $\|u\|_{\dot{L}^p_a(\G)}\equiv
\|\R^\frac{a}{\nu}u\|_{{L}^{p}(\G)}$,
for 
\begin{equation}\label{CinfGN}
A=\inf_{q:\, p<q<\frac{Qp}{Q-ap}}
(C_{GN,\mathcal{R},a ,p,q})^{\frac{ap^2}{Q(q-p)}}\,,
\end{equation}
where $C_{GN,\mathcal{R},a ,p,q}$ is the constant given in \eqref{bestcongn} for $a_1=a,a_2=0$.

In particular, for all $\|u\|_{L^p(\G)}=1$ we have 
\begin{equation}\label{a2=0u1}
\int_{\mathbb{G}} |u|^{p}\log |u| dx\leq \frac{Q}{ap^2}\log\left(A \int_\G |\R^{\frac{a}{\nu}}u(x)|^p dx \right).
\end{equation}

\end{cor}

\begin{proof}
In the  $a_{2}=0$ version of Theorem \ref{THM:GNlog}, for $p<q<\frac{Qp}{Q-ap}$, we get
\begin{equation}
    \begin{split}
        &\int_{\mathbb{G}}\frac{|u|^{p}}{\|u\|^{p}_{L^{p}(\mathbb{G})}}\log\left(\frac{|u|^{p}}{\|u\|^{p}_{L^{p}(\mathbb{G})}}\right)dx\leq \frac{q}{q-p}\log\left(\frac{\|u\|^{p}_{L^{q}(\mathbb{G})}}{\|u\|^{p}_{L^{p}(\G)}}\right)\\&
\leq\frac{q}{q-p}\log\left(C^{\frac{p}{q}}_{GN,\mathcal{R}_{1},a_{1},p,q}\frac{\left(\int_{\G}|\R_{1}^{\frac{a_{1}}{\nu_{1}}} u(x)|^{p}dx\right)^{\frac{Q(q-p)}{a_{1}pq}}\left(\int_{\G}|u(x)|^{p}dx\right)^{\frac{a_{1}pq-Q(q-p)}{a_{1}pq}}}{\|u\|^{p}_{L^{p}(\mathbb{G})}}\right)\\&
=\frac{q}{q-p}\log\left(C^{\frac{p}{q}}_{GN,\mathcal{R}_{1},a_{1},p,q}\frac{\|u\|^{\frac{Q(q-p)}{a_{1}q}}_{\dot{L}^{p}_{a_{1}}(\G)}\|u\|^{\frac{a_{1}pq-Q(q-p)}{a_{1}q}}_{L^{p}(\G)}}{\|u\|^{p}_{L^{p}(\mathbb{G})}}\right)\\&
=\frac{q}{q-p}\log\left(C^{\frac{p}{q}}_{GN,\mathcal{R}_{1},a_{1},p,q}\frac{\|u\|^{\frac{Q(q-p)}{a_{1}q}}_{\dot{L}^{p}_{a_{1}}(\G)}}{\|u\|^{\frac{Q(q-p)}{a_{1}q}}_{L^{p}(\G)}}\right).
    \end{split}
\end{equation}
Denoting $a=a_1$, we have
\begin{multline*}
\frac{q}{q-p}\log\left(C^{\frac{p}{q}}_{GN,\mathcal{R}_{1},a,p,q}\frac{\|u\|^{\frac{Q(q-p)}{a q}}_{\dot{L}^{p}_{a_{1}}(\G)}}{\|u\|^{\frac{Q(q-p)}{a q}}_{L^{p}(\G)}}\right)=
\frac{q}{q-p}\log\left(C^{\frac{p}{q}\cdot\frac{aqp}{Q(q-p)}}_{GN,\mathcal{R}_{1},a ,p,q}\frac{\|u\|^{p}_{\dot{L}^{p}_{a}(\G)}}{\|u\|^{p}_{L^{p}(\G)}}\right)^{\frac{Q(q-p)}{a q p}} \\
\leq \frac{Q}{ap}
\log\left(C^{\frac{ap^2}{Q(q-p)}}_{GN,\mathcal{R}_{1},a ,p,q}\frac{\|u\|^{p}_{\dot{L}^{p}_{a}(\G)}}{\|u\|^{p}_{L^{p}(\G)}}\right).
\end{multline*}
Minimising the constant over the admissible range of $q$, the proof 
is complete, with \eqref{a2=0u1} following immediately from  \eqref{a2=0}.
\end{proof}

Next, let us show  a logarithmic Caffarelli-Kohn-Nirenberg inequality on graded groups. Firstly, from \cite[Theorem 5.9]{RY18b} we recall the following Caffarelli-Kohn-Nirenberg on graded groups:
\begin{thm}
Let $\G$ be a graded Lie group of homogeneous dimension $Q$ and let
$\R$ be a positive Rockland operator of homogeneous degree $\nu$. Let $|\cdot|$ be an arbitrary homogeneous quasi-norm. Let $1<p,r<\infty$, $\delta\in(0,1]$, $q\in(0,\infty)$ with $q\leq\frac{p}{1-\delta}$ for $\delta\neq 1$. Let $0<a r<Q$ and $\beta,\gamma\in \mathbb{R}$ with $\delta q(Q-ar-\beta r)\leq r(Q+q\gamma-q\beta)$ and $\beta(1-\delta)-\delta a\leq \gamma\leq \beta(1-\delta)$. Assume that $\frac{q(\delta Q+r(\beta(1-\delta)-\gamma -a\delta))}{rQ}+\frac{q(1-\delta)}{p}=1$. Then
  there exists a positive constant $C$ such that
\begin{equation}\label{CKNgr}
    \||x|^{\gamma}u\|_{L^{q}(\G)}\leq C\|\R^{\frac{a}{\nu}}u\|^{\delta}_{L^{r}(\G)}\||x|^{\beta}u\|^{1-\delta}_{L^{p}(\G)},
\end{equation}
for all $u\in \dot{L}^{r}_{a}(\G)$.
\end{thm}

Now, in combination with the logarithmic H\"older inequality, our usual strategy implies the logarithmic  Caffarelli-Kohn-Nirenberg inequality on graded groups.

\begin{thm}\label{logCKNgrthm}
Let $\G$ be a graded Lie group of homogeneous dimension $Q$ and let
$\R$ be a positive Rockland operator of homogeneous degree $\nu$. Let $|\cdot|$ be an arbitrary homogeneous quasi-norm. Let $1<p,r<\infty$, $\delta\in(0,1]$, $q\in(p,\infty)$ with $q\leq\frac{p}{1-\delta}$ for $\delta\neq 1$. Let $0<a r<Q$ and $\beta,\gamma\in \mathbb{R}$ with $\delta q(Q-ar-\beta r)\leq r(Q+q\gamma-q\beta)$ and $\beta(1-\delta)-\delta a<\gamma\leq \beta(1-\delta)$. Assume that $\frac{q(\delta Q+r(\beta(1-\delta)-\gamma -a\delta))}{rQ}+\frac{q(1-\delta)}{p}=1$. Then we have
\begin{equation}
    \int_{\G}\frac{|x|^{\gamma p}|u|^{p}}{\||x|^{\gamma}u\|^{p}_{L^{p}(\G)}}\log\left(\frac{|x|^{\gamma p}|u|^{p}}{\||x|^{\gamma}u\|^{p}_{L^{p}(\G)}}\right)dx\leq \frac{q}{q-p}\log\left(C\frac{\|\R^{\frac{a}{\nu}}u\|^{\delta p}_{L^{r}(\G)}\||x|^{\beta}u\|^{(1-\delta)p}_{L^{p}(\G)}}{\||x|^{\gamma}u\|^{p}_{L^{p}(\G)}}\right),
\end{equation}
for all non-zero $|x|^{\gamma}u,|x|^{\beta}u, \R^{\frac{a}{\nu}}u\in L^{p}(\G)$.
\end{thm}
\begin{proof} The proof is immediate:
by using $q>p$ and logarithmic H\"{o}lder's inequality, we have 
\begin{equation}
    \begin{split}
        &\int_{\mathbb{G}}\frac{|x|^{\gamma p}|u|^{p}}{\||x|^{\gamma }u\|^{p}_{L^{p}(\mathbb{G})}}\log\left(\frac{|x|^{\gamma p}|u|^{p}}{\||x|^{\gamma }u\|^{p}_{L^{p}(\mathbb{G})}}\right)dx\leq \frac{q}{q-p}\log\left(\frac{\||x|^{\gamma }u\|^{p}_{L^{q}(\mathbb{G})}}{\||x|^{\gamma }u\|^{p}_{L^{p}(\G)}}\right)\\&
\stackrel{\eqref{CKNgr}}\leq\frac{q}{q-p}\log\left(C\frac{\|\R^{\frac{a}{\nu}}u\|^{\delta p}_{L^{r}(\G)}\||x|^{\beta}u\|^{(1-\delta)p}_{L^{p}(\G)}}{\||x|^{\gamma}u\|^{p}_{L^{p}(\G)}}\right),
    \end{split}
\end{equation}
completing the proof.
\end{proof}
\begin{cor}
Under the same assumptions of Theorem \ref{logCKNgrthm} with $\delta=1$ and $\gamma=0$, we get the log-Sobolev inequalities \eqref{LogSobolev} and \eqref{a2=0}.
\end{cor}

Our next aim is to give a weighted version of the log-Sobolev inequality. For this let us recall the so-called Hardy-Sobolev family of inequalities from 
 \cite[Theorem 5.1]{RY18b}, interpolating between the Hardy and the Sobolev inequalities.
 
\begin{thm}[Hardy-Sobolev inequality]\label{THM:HSgraded}
Let $\G$ be a graded Lie group of homogeneous dimension $Q$ and let
$\R$ be a positive Rockland operator of homogeneous degree $\nu$. Let $|\cdot|$ be an arbitrary homogeneous quasi-norm. Let $1<p\leq q<\infty$. Let $0<a p<Q$ and $0\leq \beta< Q$. Assume that $\frac{1}{p}-\frac{1}{q}=\frac{a}{Q}-\frac{\beta}{qQ}$. Then there exists a positive constant $C$ such that
\begin{equation}\label{EQ:HSgrad}
    \left\|\frac{u}{|x|^{\frac{\beta}{q}}}\right\|_{L^{q}(\G)}\leq C\|\R^{\frac{a}{\nu}}u\|_{L^{p}(\G)},
\end{equation}
holds for all $u\in \dot{L}^{p}_{a}(\G)$.
\end{thm}

Now let us present the weighted logarithmic Sobolev inequality on graded groups.
\begin{thm}\label{wsobingrthm}
Let $\G$ be a graded Lie group of homogeneous dimension $Q$ and let
$\R$ be a positive Rockland operator of homogeneous degree $\nu$. Let $|\cdot|$ be an arbitrary homogeneous quasi-norm. Let $1<p<\infty$ and $0\leq \beta <a p<Q$. Assume that $\frac{1}{p}-\frac{1}{q}=\frac{a}{Q}-\frac{\beta}{qQ}$. Then we have 
    \begin{equation}\label{harsobingr}
         \int_{\mathbb{G}}\frac{|x|^{-\frac{\beta p}{q}}|u|^{p}}{\||x|^{-\frac{\beta}{q}} u\|^{p}_{L^{p}(\G)}}\log\left(\frac{|x|^{-\frac{\beta p}{q}}|u|^{p}}{\||x|^{-\frac{\beta}{q}} u\|^{p}_{L^{p}(\G)}}\right) dx \leq \frac{Q-\beta}{ap-\beta}\log\left(C\frac{\|\R^{\frac{a}{\nu}}u\|^{p}_{L^{p}(\G)}}{\||x|^{-\frac{\beta}{q}} u\|^{p}_{L^{p}(\G)}}\right),
     \end{equation}
     for all nontrivial $|x|^{-\frac{\beta}{q}}|u|,\R^{\frac{a}{\nu}}u\in L^{p}(\G)$.
\end{thm}

\begin{proof}
By assumption $q=\frac{(Q-\beta)p}{Q-ap}$ and by using $0\leq \beta <ap< Q$, we have 
\begin{equation}
    q-p=\frac{(Q-\beta)p}{Q-ap}-p=\frac{p(Q-\beta-Q+ap)}{Q-ap}=\frac{(ap-\beta)p}{Q-ap}>0.
\end{equation}
It means we have $q>p$ and by using this fact, we compute
 \begin{equation}\label{vychetstploghsgr}
     0<\frac{q}{q-p}=\frac{\frac{(Q-\beta)p}{Q-ap}}{\frac{(Q-\beta)p}{Q-ap}-p}
     =\frac{\frac{Q-\beta}{Q-ap}}{\frac{(Q-\beta)}{Q-ap}-1}
     =\frac{Q-\beta}{ap-\beta}.
 \end{equation}
 We also note that we have
 $\frac{1}{p}-\frac{1}{q}= \frac{a}{Q}-\frac{\beta}{Qq}$ in Theorem \ref{THM:HSgraded}.
 By using the logarithmic H\"{o}lder's inequality in Lemma \ref{holder} with  $q>p$, we have
 \begin{equation*}
     \begin{split}
         \int_{\mathbb{G}}\frac{|x|^{-\frac{\beta p}{q}}|u|^{p}}{\||x|^{-\frac{\beta}{q}} u\|^{p}_{L^{p}(\G)}}\log\left(\frac{|x|^{-\frac{\beta p}{q}}|u|^{p}}{\||x|^{-\frac{\beta}{q}} u\|^{p}_{L^{p}(\G)}}\right) &dx\leq \frac{q}{q-p}\log\left(\frac{\||x|^{-\frac{\beta}{q}} u\|^{p}_{L^{q}(\G)}}{\||x|^{-\frac{\beta}{q}} u\|^{p}_{L^{p}(\G)}}\right)\\&
         \stackrel{\eqref{EQ:HSgrad}}\leq \frac{q}{q-p}\log\left(C\frac{\|\R^{\frac{a}{\nu}}u\|^{p}_{L^{p}(\G)}}{\||x|^{-\frac{\beta}{q}} u\|^{p}_{L^{p}(\G)}}\right)\\&
         \stackrel{\eqref{vychetstploghsgr}}=\frac{Q-\beta}{ap-\beta}\log\left(C\frac{\|\R^{\frac{a}{\nu}}u\|^{p}_{L^{p}(\G)}}{\||x|^{-\frac{\beta}{q}} u\|^{p}_{L^{p}(\G)}}\right),
     \end{split}
 \end{equation*}
completing the proof.
\end{proof}
\begin{rem}
In particular for $\beta=0$, from \eqref{harsobingr}, we get the logarithmic Sobolev inequality. To be able to use the logarithmic H\"{o}lder's inequality, we need $1<p<q<\infty$ and if  $\beta=aq$, then we have $p=q$, and so we can not obtain the logarithmic Hardy inequality by these arguments. Nevertheless, we are still able to prove a version of the logarithmic Hardy inequalities, but by a very different argument, so this will be a subject of another paper \cite{CKR21}. 
\end{rem}

\section{Logarithmic inequalities on general Lie groups}
\label{SEC:generalLie}

In this section, we present logarithmic Sobolev and Hardy-Sobolev type inequalities on general Lie groups. In such generality we have a more limited track of constants compared e.g. to the case of graded groups. Moreover, we do not obtain estimates in homogeneous Sobolev spaces but in inhomogeneous ones, associated to H\"ormander's sums of squares with a drift, to make such sub-Laplacian operators self-adjoint, to include the cases when the group is not unimodular. However, if we decide to work with sub-Laplacians and we are not after best constants, very general results are possible, also including weighted versions of these logarithmic estimates. 

Let first $\G$ be a noncompact connected Lie group. Assume that $e$ is the identity element of $\G$, and
let $X = \{X_1, . . . , X_n\}$ be a family of linearly independent, left-invariant vector fields on this group satisfying H\"{o}rmander’s condition.  We denote by $\mu_{\chi}$ a measure on $\G$, whose density is the continuous positive character $\chi$ of $\G$ with respect to the right Haar measure $\rho$ of $\G$, i.e. $d\mu_{\chi} = \chi d\rho.$ Let us denote by $\delta$ the modular function on $\G$, so that $d\lambda = \delta d\rho$ is the left Haar measure on $\G$. In \cite{HMM05}, the authors showed that the  differential operator
\begin{equation}
    \Delta_{\chi}=-\sum_{i=1}^{n}\left(X_{j}^{2}+c_{j}X_{j}\right),
\end{equation}
with domain $C_{0}^{\infty}(\G)$ on $\G$ is essentially self-adjoint on $L^{2}(\mu_{\chi})$, where $c_{j} = (X_{j}\chi)(e)$, $j = 1, \ldots , n.$

The Sobolev spaces on $\G$ are of the following form:
\begin{equation}
    L^{p}_{a}(\mu_{\chi}):=\{u:u\in L^{p}(\mu_{\chi}),\Delta^{\frac{a}{2}}_{\chi}u\in L^{p}(\mu_{\chi}) \},
\end{equation}
endowed with norm
\begin{equation}
    \|u\|_{L^{p}_{a}(\mu_{\chi})}:=\|u\|_{L^{p}(\mu_{\chi})}+\|\Delta^{\frac{a}{2}}_{\chi}u\|_{L^{p}(\mu_{\chi})}.
\end{equation}
Similarly, the Sobolev spaces with the left Haar measure $d\lambda$ (so that we take $\chi=\delta$ in this case) can be introduced by
\begin{equation}
    L^{p}_{a}(\lambda):=\{u:u\in L^{p}(\lambda),\Delta^{\frac{a}{2}}_{\chi}u\in L^{p}(\lambda) \},
\end{equation}
endowed with the norm
\begin{equation*}
    \|u\|_{L^{p}_{a}(\lambda)}:=\|u\|_{L^{p}(\lambda)}+\|\Delta^{\frac{a}{2}}_{\chi}u\|_{L^{p}(\lambda)}.
\end{equation*}
The embedding properties of the Sobolev spaces $L^{p}_{a}(\mu_{\chi})$ were established in \cite{BPTV19} and \cite{RY18a} by different methods, the latter paper giving also a family of weighted Hardy-Sobolev inequalities for these spaces and including the case of compact Lie groups.

We recall that $X = \{X_1, . . . , X_n\}$ (which is a family of linearly independent, left-invariant vector fields on this group satisfying H\"{o}rmander’s condition) induce the Carnot-Carath\'{e}odory distance $d_{CC}(\cdot, \cdot)$ (shortly, CC-distance). Assume that $B(c_{B},r_{B})$ is a ball with respect to the CC-distance centered at $c_{B}$ with radius $r_{B}$. Also, for simplicity of notation, we will write 
\begin{equation}\label{EQ:dcc}
|x|_{CC}=d_{CC}(e,x).
\end{equation}
Let us denote by $V(r):=\rho(B_{r})$ the volume of the ball $B_{r}:=B(e,r)$ with respect to the right Haar measure, then it is well-known that there exist two constants $d=d(\G,X)$ and $D=D(\G)$ such that 
\begin{equation*}
    V(r)\approx r^{d},\,\,\,\,\forall r\in (0,1],
\end{equation*}
and
\begin{equation*}
    V(r)\lesssim e^{Dr},\,\,\,\,\forall r\in (1,\infty).
\end{equation*}
We say that $d$ and $D$ are local and global dimensions of the metric measure space $(\G,d_{CC}, \rho)$, respectively. 

If $\G$ is a compact Lie group, it is unimodular, so that we have that the modular function $\delta=1$, and it means that $d\lambda=d\rho$. In addition, if we take the character $\chi=1$, we note that the Sobolev space $L^{p}_{a}(\mu_{\chi})$ coincides with Sobolev space defined by $\Delta_\G=\sum\limits_{i=1}^{n}X_{i}^{2}$. In this case $d$ is the Hausdorff dimension associated to the sub-Laplacian $\Delta_\G.$

Let us now recall the Hardy-Sobolev inequality from \cite{RY18a}, for general connected Lie groups. As before, $|x|_{CC}=d_{CC}(x,e)$ is the Carnot-Carath\'{e}odiry distance between $x$ and $e$.

\begin{thm}\label{Harsobinnthm}
Let $\G$ be a connected Lie group. Let $0\leq \beta < d$ and $1< p,q <\infty$.
\begin{itemize}
    \item[(a)] If $0< a < d$, then we have 
  \begin{equation}\label{HSinnon}
      \left\|\frac{u}{|x|_{CC}^{\frac{\beta}{q}}}\right\|_{L^{q}(\lambda)}\leq C \|u\|_{L^{p}_{a}(\lambda)},
  \end{equation}
  for all $q\geq p$ such that $\frac{1}{p}-\frac{1}{q}\leq \frac{a}{d}-\frac{\beta}{dq}$;
  \item[(b)] If $\frac{d}{p}\leq a$, then \eqref{HSinnon} holds for all $q\geq p$.
\end{itemize}
\end{thm}
Although in \cite{RY18a}, this theorem was formulated for $a<d$ in Part (b), it is clear that if \eqref{HSinnon} holds for some $a$, it also holds for all larger $a$.

Now let us present the logarithmic weighted Sobolev inequality, for general Lie groups (compact and non-compact).

\begin{thm}\label{loghardysob}
Let $\mathbb{G}$ be a connected Lie group. 
 \begin{itemize}
     \item[(a)] If $1< p <\infty$ and  $a>0$ is such that $0\leq \beta <ap< d$,  
     then we have
     \begin{equation}\label{harsobinnon1}
         \int_{\mathbb{G}}\frac{|x|_{CC}^{-\frac{\beta p}{q}}|u|^{p}}{\||x|_{CC}^{-\frac{\beta}{q}} u\|^{p}_{L^{p}(\lambda)}}\log\left(\frac{|x|_{CC}^{-\frac{\beta p}{q}}|u|^{p}}{\||x|_{CC}^{-\frac{\beta}{q}} u\|^{p}_{L^{p}(\lambda)}}\right) d\lambda(x) \leq \frac{d-\beta}{ap-\beta}\log\left(C\frac{\|u\|^{p}_{L^{p}_{a}(\lambda)}}{\||x|_{CC}^{-\frac{\beta}{q}} u\|^{p}_{L^{p}(\lambda)}}\right),
     \end{equation}
     where $q=\frac{(d-\beta)p}{d-ap}$.
     \item[(b)] If $a\geq\frac{d}{p}$, $0\leq \beta<d$ and $q>p>1$, then we have 
     \begin{equation}\label{harsobinnon2}
         \int_{\mathbb{G}}\frac{|x|_{CC}^{-\frac{\beta p}{q}}|u|^{p}}{\||x|_{CC}^{-\frac{\beta}{q}} u\|^{p}_{L^{p}(\lambda)}}\log\left(\frac{|x|_{CC}^{-\frac{\beta p}{q}}|u|^{p}}{\||x|_{CC}^{-\frac{\beta}{q}} u\|^{p}_{L^{p}(\lambda)}}\right) d\lambda(x) \leq \frac{q}{q-p}\log\left(C\frac{\|u\|^{p}_{L^{p}_{a}(\lambda)}}{\||x|_{CC}^{-\frac{\beta}{q}} u\|^{p}_{L^{p}(\lambda)}}\right).
     \end{equation}
 \end{itemize}
 The constant $C$ in these inequalities may depend on $\G$, on the system of H\"ormander's vector fields, on $p,q,a,\beta$, but not on $u$. 
\end{thm}
\begin{proof}

 (a) Let us choose $q=\frac{(d-\beta)p}{d-ap}$, then by assumption $0\leq \beta <ap< d$, we have $q>p$ and by using this fact, we compute
 \begin{equation}\label{vychetstploghs}
     0<\frac{q}{q-p}=\frac{\frac{(d-\beta)p}{d-ap}}{\frac{(d-\beta)p}{d-ap}-p}
     =\frac{\frac{d-\beta}{d-ap}}{\frac{(d-\beta)}{d-ap}-1}
     =\frac{d-\beta}{ap-\beta}.
 \end{equation}
 We also note that we have
 $\frac{1}{p}-\frac{1}{q}= \frac{a}{d}-\frac{\beta}{dq}$ in Theorem \ref{Harsobinnthm}, (a).
 By using logarithmic H\"{o}lder's inequality with $|x|_{CC}^{-\frac{\beta}{q}}|u|\in L^{p}(\lambda)\cap L^{q}(\lambda)$ and $q>p$, we get
 \begin{equation}
     \begin{split}
         \int_{\mathbb{G}}\frac{|x|_{CC}^{-\frac{\beta p}{q}}|u|^{p}}{\||x|_{CC}^{-\frac{\beta}{q}} u\|^{p}_{L^{p}(\lambda)}}\log\left(\frac{|x|_{CC}^{-\frac{\beta p}{q}}|u|^{p}}{\||x|_{CC}^{-\frac{\beta}{q}} u\|^{p}_{L^{p}(\lambda)}}\right) &d\lambda(x)\leq \frac{q}{q-p}\log\left(\frac{\||x|_{CC}^{-\frac{\beta}{q}} u\|^{p}_{L^{q}(\lambda)}}{\||x|_{CC}^{-\frac{\beta}{q}} u\|^{p}_{L^{p}(\lambda)}}\right)\\&
         \stackrel{(a)\eqref{HSinnon}}\leq \frac{q}{q-p}\log\left(C\frac{\|u\|^{p}_{L^{p}_{a}(\lambda)}}{\||x|_{CC}^{-\frac{\beta}{q}} u\|^{p}_{L^{p}(\lambda)}}\right)\\&
         \stackrel{\eqref{vychetstploghs}}=\frac{d-\beta}{ap-\beta}\log\left(C\frac{\|u\|^{p}_{L^{p}_{a}(\lambda)}}{\||x|_{CC}^{-\frac{\beta}{q}} u\|^{p}_{L^{p}(\lambda)}}\right).
     \end{split}
 \end{equation}

(b) Similarly to the previous case (a) we easily show \eqref{harsobinnon2} by using Theorem \ref{Harsobinnthm}, Part  (b).
\end{proof}

\begin{rem}\label{REM:CCR}
According to \cite[Remark 1.2]{RY18a}, the statement of Theorem \ref{Harsobinnthm} also holds with the Carnot-Carath\'eodory distance replaced by the Riemannian distance. Consequently, 
Theorem \ref{loghardysob} and Theorem \ref{THM:Nash-gen} also hold with the Carnot-Carath\'eodory distance replaced by the Riemannian distance.
\end{rem}

Under assumptions of Theorem \ref{loghardysob} with $\beta=0$, we get the logarithmic Sobolev inequality in the following form:

\begin{cor}\label{logsobcor}
Let $\mathbb{G}$ be a connected Lie group with local dimension $d$. 
Then we have the following logarithmic Sobolev inequalities:
 \begin{itemize}
     \item[(a)] If $1< p <\infty$ and $0<a<\frac{d}{p}$, then we have
     \begin{equation}\label{sobinnon1}
         \int_{\mathbb{G}}\frac{|u|^{p}}{\|u\|^{p}_{L^{p}(\lambda)}}\log\left(\frac{|u|^{p}}{\|u\|^{p}_{L^{p}(\lambda)}}\right) d\lambda(x) \leq \frac{d}{ap}\log\left(C\frac{\|u\|^{p}_{L^{p}_{a}(\lambda)}}{\|u\|^{p}_{L^{p}(\lambda)}}\right),
     \end{equation}
     with $q=\frac{dp}{d-ap}$.
     \item[(b)] If $a\geq\frac{d}{p}$ and $q>p>1$, then we have 
     \begin{equation}\label{sobinnon2}
         \int_{\mathbb{G}}\frac{|u|^{p}}{\|u\|^{p}_{L^{p}(\lambda)}}\log\left(\frac{|u|^{p}}{\|u\|^{p}_{L^{p}(\lambda)}}\right) d\lambda(x) \leq \frac{q}{q-p}\log\left(C\frac{\|u\|^{p}_{L^{p}_{a}(\lambda)}}{\|u\|^{p}_{L^{p}(\lambda)}}\right),
     \end{equation}
      with the constant $C$ depending on $\G$, on the system of H\"ormander's vector fields, on $p,q,a$, but not on $u$. 
 \end{itemize}
\end{cor}

\begin{rem}
In Corollary \ref{logsobcor}, if $\G$ is non-compact, the constant $C$ in \eqref{sobinnon1} can be chosen to be of the form $C=A_{1}S(p)^p$, where  $A_{1}=A_{1}(\G,X)>0$ depends on the group $\G$ and on the H\"ormander's system of vector fields $X$, and $S(p)=\min\left\{\frac{\left(\frac{dp}{d-ap}\right)^{1-\frac{1}{p}}}{p-1},\left(\frac{p}{p-1}\right)^{\frac{1}{p}-\frac{a}{d}}(\frac{dp-d+ap}{d-ap})\right\}$. This follows from the constant presented in \cite[Theorem 3.1]{BPV21}, by taking 
$S(p)$ to be $S(p,q)$ in that paper, with
$q=\frac{dp}{d-ap}$. This condition comes from the equality $\frac{1}{p}-\frac{1}{q}= \frac{a}{d}$ with $\beta=0$ in Theorem \ref{Harsobinnthm}, (a).
In particular, it means that the constant $C$ in \eqref{sobinnon1} depends on the group structure, the H\"{o}r\-man\-der system $X$, and $p$.
\end{rem}

We do not have such argument for the constant $C$ in \eqref{sobinnon2}, but, by letting $q\to\infty$ we can make the constant in front of the logarithm as close to 1 as we want, namely: 

\begin{cor}\label{COR:apq}
Let $\mathbb{G}$ be a connected Lie group with local dimension $d$. Let $1<p<\infty$ and $a\geq\frac{d}{p}.$ Then for every $\epsilon>0$ there exists $C_\epsilon>0$ such that for all $u\not=0$ we have 
  \begin{equation}\label{sobinnon2-eps}
         \int_{\mathbb{G}}\frac{|u|^{p}}{\|u\|^{p}_{L^{p}(\lambda)}}\log\left(\frac{|u|^{p}}{\|u\|^{p}_{L^{p}(\lambda)}}\right) d\lambda(x) \leq (1+\epsilon)\log\left(C_\epsilon\frac{\|u\|^{p}_{L^{p}_{a}(\lambda)}}{\|u\|^{p}_{L^{p}(\lambda)}}\right).
     \end{equation}
\end{cor}

\section{Nash inequality}

In this section we discuss the Nash inequality on general Lie groups, and on graded groups, and its immediate application to the decay of solutions to the heat equation for the sub-Laplacian on stratified groups. For stratified Lie groups, the Nash inequality would take the form
\begin{equation}
    \label{Nas}
    \|u\|_{L^2(\G)}^{2+\frac{4}{Q}} \leq C
    \|u\|_{L^1(\G)}^{\frac{4}{Q}}
    \|\nabla_H u\|_{L^2(\G)}^2. 
\end{equation}
For the Euclidean $\G=\mathbb R^n$ and $\nabla_H=\nabla$ being the full gradient, the inequality \eqref{Nas} was established by Nash in \cite{Nas58}, with a simple proof using the Euclidean Fourier analysis, credited to E. M. Stein. It is one of the very useful inequalities for problems in partial differential equations, and can viewed as the limiting $L^1$-case of the family of Gagliardo-Nirenberg inequalities. The best constant $C$ in \eqref{Nas} was determined by Carlen and Loss \cite{CL}, see also 
\cite{BDS20} for the further analysis of its extremisers and several different approaches. For sums of squares on unimodular Lie groups, equivalent conditions for the Nash inequality have been investigated e.g. in \cite[Theorem II.5.2]{VSCC93}, in particular, it implies a number of estimates for the heat operator $e^{-t\Delta_\G}$, see e.g. \cite[Theorem II.3.2]{VSCC93}. We refer to Davies \cite{Dav89} for a general setting and  relations to the hypercontractivity of the heat semigroup.

Let us show that \eqref{Nas} also allows for an elementary proof by a combination of H\"older and Sobolev inequalities. First, the H\"older inequality implies the inequality
\begin{equation}
    \label{EQ:Nash-Holder}
    \|u\|_{L^2(\G)}^2\leq \|u\|_{L^1(\G)}^{\frac{4}{Q+2}}
    \|u\|_{L^{\frac{2Q}{Q-2}}(\G)}^{\frac{2Q}{Q+2}}.
\end{equation}
Consequently, using Theorem \ref{GN}, we have that 
\begin{equation}
    \label{EQ:Nash-Holder2}
    \|u\|_{L^{\frac{2Q}{Q-2}}(\G)}^{2}\leq 
    C_{S}\|\mathcal{R}^{\frac{1}{\nu}} u\|_{L^2(\G)}^2,
\end{equation}
for any Rockland operator $\mathcal{R}$ of degree $\nu$,  and 
the constant $C_{S}=C^{\frac{Q-2}{Q}}_{GN,\mathcal{R},1,0,2,\frac{2Q}{Q-2}}$ in \eqref{bestcongn}.
Consequently, taking $\mathcal{R}=-\Delta_\G$ to be a sub-Laplacian on $\G$, and using the identity \eqref{lapfor}, we obtain \eqref{Nas}. The above argument also gives \eqref{Nash-final2t} with $a=1$.

On the Euclidean space, again it is well-known that the Nash inequality is equivalent to the $L^2$-log-Sobolev inequality. This is also the case in our setting on general graded Lie groups. In fact, our proof of Corollary \ref{COR:Lpspb123} shows how Nash inequality would imply the $L^2$-log-Sobolev inequality. The converse argument is given in the proof of the next theorem, which gives the Nash inequality as its special case with $a=1$. In particular, if $\G$ is a stratified group, the following theorem with $a=1$ and $\mathcal{R}$ being the sub-Laplacian implies \eqref{Nas} since 
\begin{equation}\label{lapfor}
\|(-\Delta_\G)^{1/2} u\|_{L^2(\G)}=\||\nabla_{H}| u\|_{L^2(\G)}=\|\nabla_{H}u\|_{L^2(\G)}\,,
\end{equation}
where $\nabla_{H}$ denotes the horizontal gradient on $\G$.
\begin{thm}\label{THM:Nash}
Let $\G$ be a graded group with homogeneous dimension $Q$ and let $\mathcal{R}$ be a positive Rockland operator of homogeneous degree $\nu$. Let $0\leq \beta<2a<Q.$ Then, Nash's inequality is given by 
 \begin{equation}\label{Nash-final2t}
    ||\cdot|^{-\frac{\beta(Q-2a)}{2(Q-\beta)}}u\|_{L^2(\G)}^{2+\frac{2(2a-\beta)}{Q-\beta}} \leq C \||\cdot|^{-\frac{\beta(Q-2a)}{2(Q-\beta)}}u\|_{L^1(\lambda)}^{\frac{2(2a-\beta)}{Q-\beta}}\|u\|^{2}_{\dot{L}^2_a(\lambda)},
\end{equation}
 where $C>0$ is independent of $u$.
 In particular, for $\beta=0$ and $0<a<\frac{Q}{2}$, we have
\begin{equation}\label{Nash.g}
 \|u\|_{L^2(\G)}^{2+\frac{4a}{Q}} \leq A_2 \|u\|_{L^1(\G)}^{\frac{4a}{Q}}\|u\|^{2}_{\dot{L}^2_a(\G)}\,,
 \end{equation}
 where $A_2$ is considered as in \eqref{CinfGN} for $p=2$.
\end{thm}
\begin{proof}
Let us choose $q=\frac{2(Q-\beta)}{Q-2a}$ and $v(x)=|x|^{-\frac{\beta}{q}}u(x).$ Observe that $|v|^2 \|v\|^{-2}_{L^2(\G)}\,dx$ is a probability measure on $\G$. Under this observation, an application of Jensen's inequality to the convex function $h(u)=\log \frac{1}{u}$  yields 
 \begin{eqnarray*}
\log \left(\frac{\|v\|^{2}_{L^2(\G)}}{\|v\|_{L^1(\G)}} \right) & = & h \left( \frac{\|v\|_{L^1(\G)}}{\|v\|^{2}_{L^2(\G)}}\right)=h \left( \int_{\G}\frac{1}{|v|}|v|^2 \|v\|^{-2}_{L^2(\G)}\,dx\right)\\
& \leq & \int_{\G} h \left( \frac{1}{|v|}\right) |v|^2 \|v\|^{-2}_{L^2(\G)}\,dx=\int_{\G} \frac{|v|^2}{\|v\|^{2}_{L^2(\G)}}\log |v| \,dx\,.
 \end{eqnarray*}
 The latter, combined with the weighted log-Sobolev inequality in Theorem \ref{wsobingrthm} for $p=2$, using $Q>2a$, i.e., with the estimate 
 \[
 \int_{\G} \frac{|v|^2}{\|v\|^{2}_{L^2(\G)}}\log |v| \,dx \leq \log \|v\|_{L^2(\G)}+\frac{Q-\beta}{2(2a-\beta)} \log \left( C \frac{\|u\|^{2}_{\dot{L}^2_a(\G)}}{\|v\|^{2}_{L^2(\G)}}\right)\,,
 \]
  using properties of the logarithm, gives 
 \[
 \left(1+\frac{Q-\beta}{2a-\beta} \right) \log \|v\|_{L^2(\G)} \leq \frac{Q-\beta}{2(2a-\beta)} \log \left(C \|v\|_{L^1(\G)}^{\frac{2(2a-\beta)}{Q-\beta}}\|u\|^{2}_{\dot{L}^2_a(\G)} \right)\,,
 \]
 which is equivalent to 
 \begin{equation}\label{Nash-final2}
 \||x|^{-\frac{\beta(Q-2a)}{2(Q-\beta)}}u\|_{L^2(\G)}^{2+\frac{2(2a-\beta)}{Q-\beta}} \leq C \||x|^{-\frac{\beta(Q-2a)}{2(Q-\beta)}}u\|_{L^1(\G)}^{\frac{2(2a-\beta)}{Q-\beta}}\|u\|^{2}_{\dot{L}^2_a(\G)}\,.
 \end{equation}
This completes the proof of \eqref{Nash-final2t}. To get the constant $A_2$ in \eqref{Nash.g}, we can do the same argument but use Corollary \ref{COR:Lpspb123} instead of Theorem \ref{wsobingrthm}, with $A_2$ is given by \eqref{CinfGN} for $p=2$.
\end{proof}

As an immediate application of the Nash inequality, one can compute the decay rate for the heat equation for the sub-Laplacian, following Nash's argument. We give a short application just to illustrate such idea, and refer to more general related analysis in the context of symmetric Markovian semigroups to e.g. \cite[Section II.5]{VSCC93}.

\begin{cor}\label{cor:par}
Let $\G$ be a stratified Lie group of homogeneous dimension $Q\geq 3$, and let $\Delta_\G$ be a sub-Laplacian on $\G$.
Let $u_0\in L^1(\G)\cap L^2(\G)$ be non-negative $u_0\geq 0$.
Then the solution $u$  to the heat equation 
\begin{equation}
    \label{heat.eq}
    \partial_t u=\Delta_\G u,\quad u(0,x)=u_0(x),
\end{equation}
satisfies
\begin{equation}
    \label{heatest}
    \|u(t,\cdot)\|_{L^2(\G)}\leq \left( \|u_0\|_{L^2(\G)}^{-\frac{4}{Q}}
    +\frac{4}{QA_2}\|u_0\|_{L^1(\G)}^{-\frac{4}{Q}}t\right)^{-\frac{Q}{4}},
\end{equation}
for all $t\geq 0,$ where $A_2$ is a constant in \eqref{Nash.g}.
\end{cor}
\begin{proof}
Let $\G$ be a stratified Lie group and let $\Delta_\G$ be a sub-Laplacian on $\G$.
Let $u$ be the solution to the heat equation \eqref{heat.eq}.
If $u_0\in L^1(\G)\cap L^2(\G)$ is non-negative $u_0\geq 0$, then also $u(t,x)\geq 0$ due to the positivity of the heat kernel $h_t$, see e.g. \cite[p. 48]{VSCC93}. Consequently, as in the Euclidean case, we have the mass conservation
$\|u(t,\cdot)\|_{L^1}=\|u_0\|_{L^1}.$ Indeed, we have
$$
\int_\G u(t,x) dx=\int_\G\int_\G h_t(x y^{-1})u_0(y) dydx=\int_\G u_0(y) dy
$$
in view of the Fubini theorem and the property $\|h_t\|_{L^1}=1.$ Now, multiplying by $u$ the heat equation \eqref{heat.eq}, and integrating over $\G$ we get 
\[
\langle \partial_t u(t,\cdot), u(t,\cdot) \rangle_{L^2(\G)}+\langle -\Delta_\G u(t,\cdot), u(t,\cdot) \rangle_{L^2(\G)}=0\,,
\]
or, in view of \eqref{lapfor},
\[
\frac{d}{dt}\|u(t,\cdot)\|_{L^2(\G)}=-2\|\nabla_H u(t,\cdot)\|_{L^2(\G)}\,.
\]
If we denote by $y(t):=\|u(t,\cdot)\|_{L^2(\G)}$, an application of Nash's inequality \eqref{Nas} (or \eqref{Nash-final2t}) gives
\[
y' \leq -2A_2^{-1}\|u_0\|_{L^1(\G)}^{-\frac{4}{Q}}y^{1+\frac{2}{Q}}\,,
\]
which, after integration on $t \geq 0$, implies that the solution to \eqref{heat.eq} satisfies the estimate 
\[
\|u(t,\cdot)\|_{L^2(\G)}\leq \left(\|u_0\|_{L^2(\G)}^{-\frac{4}{Q}}+\frac{4}{QA_2}\|u_0\|_{L^1(\G)}^{-\frac{Q}{4}}t \right)^{-\frac{Q}{4}}\,,
\]
and this finishes the proof. 
\end{proof}
 Let us now give the weighted Nash inequality on general Lie groups, where, as before, $|x|_{CC}=d_{CC}(x,e)$ is the Carnot-Carath\'eodory distance between $x$ and $e$, see \eqref{EQ:dcc} for details. However, we also note that the following estimates also hold true with the Carnot-Carath\'eodory distance replaced by the Riemannian distance, see Remark \ref{REM:CCR}.
 
 \begin{thm}\label{THM:Nash-gen}
 Let $\G$ be a connected Lie group with local dimension $d$. Assume that  
  $0\leq \beta <2a<d$. 
  Then, the weighted Nash inequality is given by 
 \begin{equation}\label{Nash-final2th}
 \||\cdot|_{CC}^{-\frac{\beta(d-2a)}{2(d-\beta)}}u\|_{L^2(\lambda)}^{2+\frac{2(2a-\beta)}{d-\beta}} \leq C \||\cdot|_{CC}^{-\frac{\beta(d-2a)}{2(d-\beta)}}u\|_{L^1(\lambda)}^{\frac{2(2a-\beta)}{d-\beta}}\|u\|^{2}_{L^2_a(\lambda)}.
 \end{equation}
 In particular, for $\beta=0$ and $0<a<\frac{d}{2}$, we have the Nash inequality
  \begin{equation}\label{Nash-final2th0}
 \|u\|_{L^2(\lambda)}^{2+\frac{4a}{d}} \leq C \|u\|_{L^1(\lambda)}^{\frac{4a}{d}}\|u\|^{2}_{L^2_a(\lambda)}.
 \end{equation}
 
 Moreover, if $a\geq\frac{d}{2}$ and $0\leq \beta<d$, then for every $q>2$ the weighted Nash inequality is given by
 \begin{equation}
     \||\cdot|_{CC}^{-\frac{\beta}{q}}u\|^{2+\frac{2(q-2)}{q}}_{L^{2}(\lambda)}\leq C\||\cdot|_{CC}^{-\frac{\beta}{q}}u\|^{\frac{2(q-2)}{q}}_{L^{1}(\lambda)}\|u\|^{2}_{L^{2}_{a}(\lambda)}.
 \end{equation}
 In particular, for $\beta=0$ and any $a\geq\frac{d}{2}$ and $q>2$, we have the Nash type inequality
  \begin{equation}
     \|u\|^{2+\frac{2(q-2)}{q}}_{L^{2}(\lambda)}\leq C\|u\|^{\frac{2(q-2)}{q}}_{L^{1}(\lambda)}\|u\|^{2}_{L^{2}_{a}(\lambda)}.
 \end{equation}
 \end{thm}
 
 \begin{proof}
 The proof of this theorem is similar to that of Theorem \ref{THM:Nash}, but instead of Theorem \ref{wsobingrthm} we use Theorem \ref{loghardysob}. 
 Then, we choose $q=\frac{2(d-\beta)}{d-2a}$. We denote $v(x)=|x|_{CC}^{-\frac{\beta}{q}}u(x).$
 
 Observe that $|v|^{2}=||x|_{CC}^{-\frac{\beta}{q}}u|^2$ and $\frac{|v|^{2}}{\|v\|^{2}_{L^{2}(\lambda)}}d\lambda(x)$ is a probability measure with respect to left Haar measure on $\G$. Under this observation, an application of Jensen's inequality to the convex function $h(v)=\log \frac{1}{v}$  yields 
 \begin{eqnarray}\label{nashjensen1}
\log \left(\frac{\|v\|^{2}_{L^2(\lambda)}}{\|v\|_{L^1(\lambda)}} \right) & = & h \left( \frac{\|v\|_{L^1(\lambda)}}{\|v\|^{2}_{L^2(\lambda)}}\right)=h \left( \int_{\G}\frac{1}{|v|}|v|^2 \|v\|^{-2}_{L^2(\lambda)}\,d\lambda(x)\right)\\
& \leq & \int_{\G} h \left( \frac{1}{|v|}\right) |v|^2 \|v\|^{-2}_{L^2(\lambda)}\,d\lambda(x)=\int_{\G} \frac{|v|^2}{\|v\|^{2}_{L^2(\lambda)}}\log |v| \,d\lambda(x)\,.\nonumber
 \end{eqnarray}
 The latter, combined with the log-Hardy-Sobolev inequality in Theorem \ref{loghardysob} for $p=2$, i.e., with the estimate 
 \[
 \int_{\G} \frac{|v|^2}{\|v\|^{2}_{L^2(\lambda)}}\log |v| \,d\lambda(x) \leq \log \|v\|_{L^2(\lambda)}+\frac{d-\beta}{2(2a-\beta)} \log \left( C \frac{\|u\|^{2}_{L^2_a(\lambda)}}{\|v\|^{2}_{L^2(\lambda)}}\right)\,,
 \]
 where $C>0$ for $p=2$, using properties of the logarithm, gives 
 \[
 \left(1+\frac{d-\beta}{2a-\beta} \right) \log \|v\|_{L^2(\lambda)} \leq \frac{d-\beta}{2(2a-\beta)} \log \left(C \|v\|_{L^1(\lambda)}^{\frac{2(2a-\beta)}{d-\beta}}\|u\|^{2}_{L^2_a(\lambda)} \right)\,,
 \]
 which is equivalent to 
 \begin{equation}\label{Nash-final23}
 \||\cdot|_{CC}^{-\frac{\beta}{q}}u\|_{L^2(\lambda)}^{\frac{2(2a-\beta)}{d-\beta}+2} \leq C \||\cdot|_{CC}^{-\frac{\beta}{q}}u\|_{L^1(\lambda)}^{\frac{2(2a-\beta)}{d-\beta}}\|u\|^{2}_{L^2_a(\lambda)}\,.
 \end{equation}
 
 Let us now prove the case $q>2$ and $a\geq\frac{d}{2}$. Similarly, let us choose $|v|^{2}=||x|_{CC}^{-\frac{\beta}{q}}u|^2$, and observe that $\frac{|v|^{2}}{\|v\|^{2}_{L^{2}(\lambda)}}d\lambda(x)$ is a probability measure with respect to left Haar measure on $\G$.
 By combining \eqref{nashjensen1} and \eqref{harsobinnon2}, we have
 \begin{equation*}\label{nashjensen2}
     \begin{split}
2\log\|v\|_{L^{2}(\lambda)}-\log\|v\|_{L^{1}(\lambda)}&=\log \left(\frac{\|v\|^{2}_{L^2(\lambda)}}{\|v\|_{L^1(\lambda)}} \right)  \\&
=  h \left( \frac{\|v\|_{L^1(\lambda)}}{\|v\|^{2}_{L^2(\lambda)}}\right)\\&=h \left( \int_{\G}\frac{1}{|v|}|v|^2 \|v\|^{-2}_{L^2(\lambda)}\,d\lambda(x)\right)\\&
\leq \int_{\G} h \left( \frac{1}{|v|}\right) |v|^2 \|v\|^{-2}_{L^2(\lambda)}\,d\lambda(x)\\&
=\int_{\G} \frac{|v|^2}{\|v\|^{2}_{L^2(\lambda)}}\log |v| \,d\lambda(x)\\& 
\stackrel{\eqref{harsobinnon2}}\leq \log\|v\|_{L^{2}(\lambda)}+\frac{q}{2(q-2)}\log\left(C\frac{\|u\|^{2}_{L^{2}_{a}(\lambda)}}{\|v\|^{2}_{L^{2}(\lambda)}}\right)\\&         
   =\frac{q}{2(q-2)}\log\left(C\|u\|^{2}_{L^{2}_{a}(\lambda)}\right)+\log\|v\|^{1-\frac{q}{q-2}}_{L^{2}(\lambda)}\\&
   =\frac{q}{2(q-2)}\log\left(C\|u\|^{2}_{L^{2}_{a}(\lambda)}\right)+\left(1-\frac{q}{q-2}\right)\log\|v\|_{L^{2}(\lambda)}.
     \end{split}
 \end{equation*}
By using this last fact, we have
\begin{equation}
    \begin{split}
        \left(1+\frac{q}{q-2}\right)\log\|v\|_{L^{2}(\lambda)}\leq \frac{q}{2(q-2)}\log\left(C\|u\|^{2}_{L^{2}_{a}(\lambda)}\|v\|^{\frac{2(q-2)}{q}}_{L^{1}(\lambda)}\right),
    \end{split}
\end{equation}
so that we have 
\begin{equation}
    \|v\|^{\frac{2(q-2)}{q}+2}_{L^{2}(\lambda)}\leq C\|u\|^{2}_{L^{2}_{a}(\lambda)}\|v\|^{\frac{2(q-2)}{q}}_{L^{1}(\lambda)}.
\end{equation}
Finally, this means
\begin{equation}
   \||\cdot|_{CC}^{-\frac{\beta}{q}}u\|^{\frac{2(q-2)}{q}+2}_{L^{2}(\lambda)}\leq C\|u\|^{2}_{L^{2}_{a}(\lambda)}\||\cdot|_{CC}^{-\frac{\beta}{q}}u\|^{\frac{2(q-2)}{q}}_{L^{1}(\lambda)}. 
\end{equation}
This completes the proof.
 \end{proof}
\section{Logarithmic Sobolev and Gross inequalities on stratified groups}

We now give some corollaries of the previous results for the setting of stratified Lie groups. First we note that for $a=1$ and $p=2$, we can rewrite \eqref{a2=0} using the horizontal gradient on $\G.$ Indeed, we use Corollary \ref{COR:Lpspb123} with the sub-Laplacian 
$\Delta_\G=-\nabla_H^*\nabla_H$.

\begin{cor}\label{Cor.Log}
Let $\G$ be a stratified Lie group with homogeneous dimension $Q$. Then, the following log-Sobolev inequality is satisfied
\begin{equation}
    \label{cor.log.sob}
    \int_{\mathbb{G}}|u|^2 \log|u|\,dx \leq \frac{Q}{4}\log \left(A \int_{\mathbb{G}}|\nabla_{H}u|^2\,dx \right)\,,
\end{equation}
for every $u \in H^1{(\G)}$ such that $\|u\|_{L^2(\G)}=1$, where \begin{equation}\label{Cinf}A=\inf_{2<q<\frac{2Q}{Q-2}}(C_{GN, -\Delta_\G,1,2,q})^{\frac{4}{Q(q-2)}}\,,
\end{equation}
where $C_{GN, -\Delta_\G,1,2,q}$ is the constant given in \eqref{bestcongn} for $a_1=1,a_2=0$ and $p=2$.
\end{cor}
\begin{proof}
Inequality \eqref{a2=0} for $p=2$, $\|u\|_{L^2(\G)}=1$, $a_1=1$, and for $\mathcal{R}_1$ being the (positive) sub-Laplacian $-\Delta_\G$ on $\G$ gives 
\[
\int_{\G}|u|^2 \log(|u|^2)\,dx \leq \frac{q}{q-2}\log \left(C_{GN, -\Delta_\G,1,2,q}^{\frac{2}{q}}\|(-\Delta_{\G})^{1/2}u\|_{L^2(\G)}^{\frac{Q(q-2)}{q}} \right)\,,
\]
for any $q\in \left(2,\frac{2Q}{Q-2}\right)$. Rewriting the above inequality as
\[
\int_{\G}|u|^2 \log(|u|^2)\,dx \leq \frac{q}{q-2}\log \left( (C_{GN, -\Delta_\G,1,2,q})^{\frac{4}{Q(q-2)}}\|(-\Delta_{\G})^{1/2}u\|_{L^2(\G)}^{2}\right)^{\frac{Q(q-2)}{2q}}\,,
\]
and applying the properties of the logarithm we get 
\begin{equation}\label{in.cor.log.sob}
\int_{\G}|u|^2 \log(|u|)\,dx \leq \frac{Q}{4}\log \left( (C_{GN, -\Delta_\G,1,2,q})^{\frac{4}{Q(q-2)}}\int_{\G}|\nabla_{H}u|^2\,dx\right)\,.
\end{equation}
Finally, the latter inequality, combined with \eqref{lapfor}, proves \eqref{cor.log.sob} for any $q\in \left(2,\frac{2Q}{Q-2}\right)$, and the proof is now complete.
\end{proof}

\begin{rem}\label{REM:density}
Let us recall an argument, following the proof of \cite[Theorem 8.14]{LL01}, that \eqref{cor.log.sob} for $C_0^\infty(\G)$ implies the same inequality for all functions in $H^1(\G)$.
To prove that \eqref{cor.log.sob} holds true for $u \in H^1(\G)$, let us consider $u \in H^1(\G)\cap L^{2+\delta}(\G)\cap L^{2-\delta}(\G)$ for some $\delta>0$. Following the ideas from \cite[Theorem 8.14]{LL01} in the Euclidean setting, we can show that the derivative $\frac{d}{dp}\|u\|_{L^{p}(\G)}^{2}|_{p=2}$ can be defined formally and we have 
\begin{equation}\label{der.p}
\frac{d}{dp}\|u\|_{L^{p}(\G)}^{2}|_{p=2}=\int_{\G}|u(x)|^2\log (|u(x)|)dx\,.
\end{equation}
 Indeed, note that since the function $p \mapsto t^p$ is convex, the function 
\[
Q(p_1,p_2)=\frac{t^{p_1}-t^{p_2}}{p_1-p_2}\,,
\]
is monotonically non-deacreasing in $p_2$ for fixed $p_1$. Hence, for $-\delta\leq \varepsilon \leq \delta$, we have 
\[
\frac{|u(x)|^{2}-|u(x)|^{2-\delta}}{\delta}\leq \frac{|u(x)|^{2}-|u(x)|^{2-\varepsilon}}{\varepsilon}\leq \frac{|u(x)|^{2+\delta}-|u(x)|^{2}}{\delta}\,.
\]
Taking the limit of $\frac{|u(x)|^{2}-|u(x)|^{2-\varepsilon}}{\varepsilon}$ as $\varepsilon \rightarrow 0$, and integrating over $\G$, the dominated convergence theorem implies that 
\[
\frac{d}{dp}\|u\|_{L^{p}(\G)}^{2}|_{p=2}=\int_{\G}|u(x)|^2 \log \left(\frac{|u(x)|}{\|u\|_{L^{2}(\mu)}}\right)dx\,,
\]
which gives \eqref{der.p} since $\|u\|_{L^2(\G)}=1$. From the above, using the density of $C_{0}^{\infty}(\G)$ in $H^{1}(\G)\cap L^{2+\delta}(\G)\cap L^{2-\delta}(\G)$ we have shown that \eqref{cor.log.sob} makes sense for $u \in H^1(\G)\cap L^{2+\delta}(\G)\cap L^{2-\delta}(\G)$, $\|u\|_{L^2(\G)}=1$. Hence, the result follows by the Sobolev embedding theorem, since for large $u(x)$, we have that $|u(x)|^2\log {|u(x)|} \lesssim |u(x)|^{2+\delta} $ for any $\delta>0.$
\end{rem}

It is well-known that in $\mathbb R^n$, the Euclidean log-Sobolev inequality implies the Gross version of such an inequality with the Gaussian measure. We now show that to some extent, this implication continues to hold also on stratified groups. However, instead of the Gaussian measure on the whole space we consider the Gaussian measure on the first stratum. In the case of the Euclidean  $\G=\mathbb R^n$, this recovers the above implication since the first stratum is the whole space.
Since in $\mathbb R^n$, we can take $A$ in \eqref{cor.log.sob} to be $A=\frac{2}{\pi e n}$, and we have $Q=n_1=n$, as should be expected, the constant $\gamma$ in \eqref{k} is
$\gamma_{\mathbb R^n}=(2\pi)^{-n/2}$, recovering the well-known Gross inequality \eqref{EQ:Gross} on $\mathbb R^n$. For the exact value of this constant on the Heisenberg group, see \cite{CKR23}.

\begin{thm}\label{semi-g}
Let $\G$ be a stratified group with homogeneous dimension $Q$, topological dimension $n$, and let $n_1$ be the dimension of the first stratum of its Lie algebra, i.e., for $x \in \G$ we can write $x=(x',x'') \in \mathbb{R}^{n_1} \times \mathbb{R}^{n-n_1}$. Then the following ``semi-Gaussian'' log-Sobolev inequality is satisfied
\begin{equation}
    \label{gaus.log.sob}
    \int_{\G}|g|^2\log|g|\,d\mu \leq \int_{\G} |\nabla_{H}g|^2\,d\mu\,,
\end{equation}
for any $g \in H^{1}(\G,\mu)$ such that $\|g\|_{L^2(\mu)}=1,$
 where $\mu=\mu_1 \otimes \mu_2$, and $\mu_1$ is the Gaussian measure on $\mathbb{R}^{n_1}$ given by $d\mu_1=\gamma e^{-\frac{|x'|^2}{2}}dx'$, for $x'\in \mathbb{R}^{n_1}$, and $|x'|$ being the Euclidean norm of $x'$, where 
 \begin{equation}
     \label{k}
     \gamma:=(4^{-1}Q e^{\frac{2n_1}{Q}{-1}}A)^{Q/2}\,,
 \end{equation}
 where $A$ is given in \eqref{Cinf},
 and $\mu_2$ is the Lebesgue measure $dx''$ on $\mathbb{R}^{n-n_1}$.
\end{thm}
\begin{proof}
Assume that $g\in C^{\infty}_{0}(\G)$ is a compactly supported function such that $\|g\|_{L^2(\mu)}=1,$ with $\mu$ as in the conditions of the theorem. Notice that the space of compactly supported functions $C^{\infty}_{0}(\G)$ is still a dense subspace of $H^1(\G,\mu), L^p(\G,\mu)$, $p \in [1,\infty)$. Let us define $f(x)$ by 
 \begin{equation}
     \label{fg}
     f(x)=\gamma^{1/2}e^{-\frac{|x'|^2}{4}}g(x),
  \end{equation}    
where $\gamma$ is as in \eqref{k}.  Then, clearly $f \in  C^{\infty}_{0}(\G) $, and  we claim that $f \in L^2(\G)$ and $\|f\|_{L^2(\G)}=1$.
Indeed, we observe that for $\mu$ as in the hypothesis we get 
 \begin{equation}
     \label{norms1}
1=\|g\|_{L^2(\mu)}^2=\int_{\G}\gamma^{-1} e^{\frac{|x'|^2}{2}}|f(x)|^2\,d\mu=\int_{\G}|f(x)|^2\,dx\,.
 \end{equation}
Applying the log-Sobolev inequality \eqref{cor.log.sob} to $f$, we estimate  
\begin{eqnarray}\label{thm.eq.glogg}
\int_{\G}|g(x)|^2\log |g(x)|\,d\mu & = & \int_{\G} |f(x)|^2 \log |\gamma^{-1/2}e^{\frac{|x'|^2}{4}}f(x)|\,dx \nonumber\\
& = & \int_\G |f(x)|^2\log |f(x)| dx \nonumber \\
& & +\log (\gamma^{-1/2})+\int_{\G}\frac{|x'|^2}{4} |f(x)|^2\,dx\,
\nonumber\\
& \leq & \frac{Q}{4}\log \left(A \int_{\G}|\nabla_{H}f(x)|^2\,dx\right) \nonumber \\
& & +\log (\gamma^{-1/2})+\int_{\G}\frac{|x'|^2}{4} |f(x)|^2\,dx\,.
\end{eqnarray}
Now, recall that for the $n_1$ elements of the first stratum of the Lie algebra of $\G$ the following formulas are known (see \cite{FR16}):
\[
X_i=\partial_{x_{i}^{'}}+\sum_{j=1}^{n-n_1}p_{j}^{i}(x')\partial_{x''_{j}}\,, \forall i=1,\ldots,n_1\,. 
\]
Hence, for each $i=1,\ldots,n_1$, we can calculate 
\begin{eqnarray}\label{Xg}
|X_ig(x)|^2 &=& \gamma^{-1} e^{\frac{|x'|^2}{2}} \left|X_if(x)+\frac{x'_{i}}{2}f(x) \right|^{2}\nonumber\\
&=& \gamma^{-1} e^{\frac{|x'|^2}{2}} \left(|X_if(x)|^2+\frac{(x'_{i})^2}{4}|f(x)|^2+{\rm Re} \overline{(X_if(x))}x'_{i}f(x) \right)\,.
\end{eqnarray}
Moreover, notice that for each $x'_{i}$,  $i=1,\ldots,n_1$, we have
\begin{eqnarray*}
{\rm Re}\int_{\G}\overline{(\partial_{x'_{i}}f(x))}x'_{i}f(x)\,dx & =& -{\rm Re}\int_{\G}(\partial_{x'_{i}}f(x))x'_{i}\overline{f(x)}\,dx-\int_{\G}|f(x)|^2\,dx\\
&=& -{\rm Re}\int_{\G}\overline{(\partial_{x'_{i}}f(x))}x'_{i}f(x)\,dx-1\,,
\end{eqnarray*}
where the last implies that 
\begin{equation}
    \label{int.parts1}
    {\rm Re}\int_{\G}\overline{(\partial_{x'_{i}}f(x))}x'_{i}f(x)\,dx=-\frac{1}{2}\,.
\end{equation}
Additionally, for $j=1,\ldots,n-n_1$ and for $i=1,\ldots,n_1$, we have
\begin{eqnarray*}
   {\rm Re} \int_{\G}p_{j}(x')\overline{(\partial_{x_{j}''}f(x))}x'_{i}f(x)& = & - {\rm Re}\int_{\G}\partial_{x_{j}''}((p_{j}(x')f(x)x'_{i})\overline{f(x)}\,dx\\
    &=&- {\rm Re}\int_{\G}p_{j}(x')\overline{(\partial_{x_{j}''}f(x))}x'_{i}f(x)\,dx\,,
\end{eqnarray*}
i.e., we have that 
\begin{equation}  \label{int.parts2}
   {\rm Re} \int_{\G}p_{j}(x')\overline{(\partial_{x_{j}''}f(x))}x'_{i}f(x)\,dx=0\,.
\end{equation}
Hence, by \eqref{int.parts1} and \eqref{int.parts2} we get 
\[
\int_{\G}{\rm Re}\overline{(X_if(x))}x'_{i}f(x)\,dx=-\frac{1}{2}\,,
\]
for each $i=1,\ldots,n_1$, so that using the expression \eqref{Xg} we get 
\begin{equation}\label{thm.eq.nablag}
\int_{\G}|\nabla_{H}g(x)|^2\,d\mu=\int_{\G}|\nabla_{H}f(x)|^2\,dx+\int_{\G}\frac{|x'|^2}{4}|f(x)|^2\,dx-\frac{n_1}{2}\,.
\end{equation}
Now, by using \eqref{thm.eq.glogg} and \eqref{thm.eq.nablag}, to show \eqref{gaus.log.sob} it suffices to show that 
\[
\frac{Q}{4}\log \left(A \int_{\G}|\nabla_{H}f(x)|^2\,dx\right)+\log (\gamma^{-1/2}) \leq \int_{\G}|\nabla_{H}f(x)|^2\,dx-\frac{n_1}{2}\,,
\]
or equivalently,
\[
\frac{Q}{4}\log \left( A \int_{\G}|\nabla_{H}f(x)|^2\,dx \right)+\log (\gamma^{-1/2}) + \log e^{\frac{n_1}{2}}\leq \int_{\G}|\nabla_{H}f(x)|^2\,dx\,.
\]
Now, we can instead show that
\[
\frac{Q}{4} \log \left(A \gamma^{-\frac{2}{Q}} e^{\frac{n_1}{2}\cdot\frac{4}{Q}} \int_{\G}|\nabla_{H}f(x)|^2\,dx  \right)\leq \int_{\G}|\nabla_{H}f(x)|^2\,dx\,,
\]
or that 
\[
\log \left(A \gamma^{-\frac{2}{Q}} e^{\frac{2n_1}{Q}} \int_{\G}|\nabla_{H}f(x)|^2\,dx  \right)\leq \frac{4}{Q} \int_{\G}|\nabla_{H}f(x)|^2\,dx\,.  
\]
Now, since $\log r \leq r-1$, for all $r>0$, it suffices to show that 
\[
{e^{-1}}A\gamma^{-\frac{2}{Q}}e^{\frac{2n_1}{Q}} \int_{\G}|\nabla_{H}f(x)|^2\,dx \leq \frac{4}{Q} \int_{\G}|\nabla_{H}f(x)|^2\,dx\,,
\]
{since then we would have
\begin{eqnarray*}
\log \left(A \gamma^{-\frac{2}{Q}} e^{\frac{2n_1}{Q}} \int_{\G}|\nabla_{H}f(x)|^2\,dx  \right) & = & \log \left(e^{-1} A\gamma^{-\frac{2}{Q}}e^{\frac{2n_1}{Q}} \int_{\G}|\nabla_{H}f(x)|^2\,dx\right)+1\\
& \leq & e^{-1} A\gamma^{-\frac{2}{Q}}e^{\frac{2n_1}{Q}} \int_{\G}|\nabla_{H}f(x)|^2\,dx\\
& = & \frac{4}{Q}  \int_{\G}|\nabla_{H}f(x)|^2\,dx,
\end{eqnarray*}}%
where the last holds as an equality for the choice of $\gamma$ as in \eqref{k}, and we have shown \eqref{gaus.log.sob} for $g \in  C^{\infty}_{0}(\G)$. 

Now, to prove that \eqref{gaus.log.sob} holds true for $g \in H^1(\G,\mu)$, we approximate $g$ in $H^1(\G,\mu)$ by a sequence of $g_n\in C_0^\infty(\G).$ We need to show that the integral $\int_\G |g_n|^2\log |g_n| dx$ has a limit.

Let us consider $f$ defined by \eqref{fg}. By \eqref{thm.eq.nablag}, we see that the operator $g\mapsto f$ is bounded from $H^1(\G,\mu)$ to $H^1(\G)$ and to $L^2(\G,|x'|^2 dx).$ Since $f_n\in C_0^\infty(\G)$, by Remark \ref{REM:density}, there exists the limit $\int_\G |f_n|^2 \log |f_n| dx.$
Also, the limit of $f_n$ exists in $L^2(\G,|x'|^2 dx).$
Therefore, using the equality in \eqref{thm.eq.glogg}, the limit of $\int_\G |g_n|^2\log |g_n| dx$ also exists.
\end{proof}

As in the Euclidean setting, cf. e.g. \cite{Bec98}, one can show that ``horizontal'' log-Sobolev inequality \eqref{cor.log.sob} is equivalent to the Gross type ``semi-Gaussian'' log-Sobolev inequality  \eqref{gaus.log.sob}. 

\begin{thm}\label{equiv.thm}
Let $\G$ and $d\mu$ be as in the hypothesis of Theorem \ref{semi-g}. The following statements are true and imply each other:
\begin{enumerate}
    \item[(i)]  for $g$ such that $\|g\|_{L^2(\mu)}=1$ we have  \begin{equation}
    \label{thm.eq.itmq}
    \int_{\G}|g|^2 \log |g|\,d\mu \leq \int_{\G} |\nabla_{H}g|^2\,d\mu\,;\end{equation}
    \item[(ii)] for $f$ such that $\|f\|_{L^2(\G)}=1$ we have \begin{equation}\label{thm.eq.itm2}\int_{\G}|f|^{2}\log |f|\,dx \leq \frac{Q}{4}\log \left(A \int_{\G} |\nabla_{H}f|^2\,dx \right)\,,\end{equation} where $A$ is as in \eqref{Cinf}.
\end{enumerate}

\end{thm}
\begin{proof}
It was already shown in Corollary \ref{Cor.Log} that (ii) holds true, and in
Theorem \ref{semi-g} that (ii) implies (i); so it is enough to show that also (i) implies (ii).
For $\gamma$ as in \eqref{k}, let us define $g$ by 
\[
g(x)=\gamma^{-\frac{1}{2}}e^{\frac{|x'|^2}{4}}f(x)\,.
\]
Assuming $\|g\|_{L^2(\mu)}=1$, then we also have that $\|f\|_{L^2(\G)}=1$. Now, a combination of
\[
\int_{\G}|g(x)|^2\log |g(x)|\,d\mu  =  \int_{\G} |f|^2 \log |\gamma^{-1/2}e^{\frac{|x'|^2}{4}}f(x)|\,dx\,,
\]
together with the equality \eqref{thm.eq.nablag} and the assumption \eqref{thm.eq.itmq} gives
\begin{equation}\label{before.eps}
\log (\gamma^{-\frac{1}{2}}e^{\frac{n_1}{2}})+\int_{\G}|f|^2\log |f|\,dx \leq \int_{\G} |\nabla_{H}f|^2\,dx\,.
\end{equation}
Now, for any $\epsilon>0$, we have $\|f_\epsilon\|_{L^2(\G)}=1$, where $f_\epsilon(x):=\epsilon^{\frac{Q}{2}}f(D_{\epsilon}(x))$, where the mapping $D_\epsilon: \G \rightarrow \G$ stands for the dilations on $\G$. Now, an application of \eqref{before.eps} gives
\begin{eqnarray}\label{after.eps}
\int_{\G}|f|^2\log |f|\,dx & \leq & \epsilon^2 \int_{\G} |\nabla_{H}f|^2\,dx-\frac{Q}{2}\log \epsilon-\log (\gamma^{-\frac{1}{2}}e^{\frac{n_1}{2}})\nonumber\\
& =& \epsilon^2 \int_{\G} |\nabla_{H}f|^2\,dx-\frac{Q}{2}\log \epsilon+\frac{Q}{4} \log \left(\frac{QA}{4e} \right)
\end{eqnarray}
    for all $\epsilon>0$. 
    Now, minimizing the expression on the right-hand side of \eqref{after.eps} over $\epsilon$ we get  
    \begin{eqnarray*}
    \int_{\G}|f|^2\log |f|\,dx & \leq & \frac{Q}{4}+\frac{Q}{4} \log \left(4 \int_{\G}|\nabla_{H}f|^2\,dx \right)- \frac{Q}{4}\log Q+\frac{Q}{4} \log \left(\frac{QA}{4e} \right)\\
    & = & \frac{Q}{4} \log \left(A \int_{\G}|\nabla_{H}f|^2\,dx \right)\,.
    \end{eqnarray*}
    This shows that \eqref{thm.eq.itmq} implies \eqref{thm.eq.itm2}. 
\end{proof}

\section{Weighted Gross logarithmic Sobolev inequality}

Note that the weighted log-Sobolev inequality also implies its Gross version. In this section we discuss this implication, as well as give examples for it for the Euclidean space, with different choices of homogeneous quasi-norms.

To start with, let us note a special case of the weighted Sobolev inequality in Theorem \ref{wsobingrthm} in the case of $p=2$ and $a=1$, also using the equality \eqref{lapfor}:
let $\G$ be a stratified Lie group of homogeneous dimension $Q$ and let $|\cdot|$ be an arbitrary homogeneous quasi-norm. Assume that $0\leq \beta <2<Q$ and $\frac{1}{p}-\frac{1}{q}=\frac{a}{Q}-\frac{\beta}{qQ}$. Then, as a consequence of Theorem \ref{wsobingrthm}, we have 
    \begin{equation}\label{harsobingr-2-f}
         \int_{\mathbb{G}}\frac{|x|^{-\frac{2\beta }{q}}|u|^{2}}{\||x|^{-\frac{\beta}{q}} u\|^{2}_{L^{2}(\G)}}\log\left(\frac{|x|^{-\frac{2\beta }{q}}|u|^{2}}{\||x|^{-\frac{\beta}{q}} u\|^{2}_{L^{2}(\G)}}\right) dx \leq \frac{Q-\beta}{2-\beta}\log\left(C\frac{\|\nabla_H u\|^{2}_{L^{2}(\G)}}{\||x|^{-\frac{\beta}{q}} u\|^{2}_{L^{2}(\G)}}\right),
     \end{equation}
     for all nontrivial $|x|^{-\frac{\beta}{q}}|u|,\nabla_Hu\in L^{p}(\G)$.
In particular, the condition on $q$ implies that 
 $q=\frac{2(Q-\beta)}{Q-2},$ so that
 $\frac{\beta}{q}=\frac{\beta(Q-2)}{2(Q-\beta)}.$
 Altogether, we obtain the following
 extension of Corollary \ref{Cor.Log}, which can be viewed as the case of $\beta=0$ of the following statement:
 \begin{cor}\label{Cor.Log-w}
Let $\G$ be a stratified Lie group with homogeneous dimension $Q$ and let $|\cdot|$ be an arbitrary homogeneous quasi-norm. Assume that $0\leq \beta <2<Q$. Then, there exists $C>0$ such that the following weighted log-Sobolev inequality holds,
  \begin{equation}\label{harsobingr-2}
         \int_{\mathbb{G}}{|x|^{-\frac{\beta(Q-2)}{Q-\beta}}|u|^{2}}\log\left({|x|^{-\frac{\beta(Q-2)}{2(Q-\beta)}}|u|}\right) dx \leq \frac{Q-\beta}{2(2-\beta)}\log\left(C\int_\G |\nabla_H u|^{2} dx\right),
     \end{equation}
for every $u$ such that $\||x|^{-\frac{\beta(Q-2)}{2(Q-\beta)}}u\|_{L^2(\G)}=1$.
\end{cor}
Consequently, the argument similar to that of Theorem \ref{semi-g} gives us its weighted version:

\begin{thm}\label{semi-g-w}
Let $\G$ be a stratified group with homogeneous dimension $Q$, topological dimension $n$, and let $n_1$ be the dimension of the first stratum of its Lie algebra, i.e., for $x \in \G$ we can write $x=(x',x'') \in \mathbb{R}^{n_1} \times \mathbb{R}^{n-n_1}$. 
Let $|\cdot|$ be an arbitrary homogeneous quasi-norm on $\G$ and let $0\leq \beta <2<Q$. 
Let $|x'|$ denote the Euclidean norm of $x'$, and let $M>0$ be a constant such that we have 
\begin{equation}\label{EQ:norms}
|x'|\leq M|x|,
\end{equation}
for the quasi-norm $|x|$, for all $x\in\G.$
Then the following weighted ``semi-Gaussian'' log-Sobolev inequality is satisfied
\begin{equation}
    \label{gaus.log.sob-w}
    \int_{\G}|x|^{-\frac{\beta(Q-2)}{Q-\beta}}|g|^2\log\left(|x|^{-\frac{\beta(Q-2)}{2(Q-\beta)}}|g|\right)\,d\mu \leq \int_{\G} |\nabla_{H}g|^2\,d\mu\,,
\end{equation}
for all $g$ such that $\||x|^{-\frac{\beta(Q-2)}{2(Q-\beta)}}g\|_{L^2(\mu)}=1,$
 where $\mu=\mu_1 \otimes \mu_2$, and $\mu_1$ is the Gaussian measure on $\mathbb{R}^{n_1}$ given by $d\mu_1=\gamma e^{-\frac{|x'|^2}{2}}dx'$, for $x'\in \mathbb{R}^{n_1}$, where the normalisation constant $\gamma$ is now given by 
 \begin{equation}
     \label{k-w}
     \gamma:= \left(\frac{Q-\beta}{2(2-\beta
     )}C e^{(n_1+\frac{M^2}{2})\frac{2-\beta}{Q-\beta}-1}
     \right)^{\frac{Q-\beta}{2-\beta}}\,,
 \end{equation} 
 and $\mu_2$ is the Lebesgue measure $dx''$ on $\mathbb{R}^{n-n_1}$.
\end{thm}
We note that the constant $M$ as in \eqref{EQ:norms} exists for any homogeneous quasi-norm $|\cdot|.$ Indeed, on its unit quasi-sphere, $|x'|$ achieves a maximum $M$, so \eqref{EQ:norms} follows by homogeneity, with $M=\max\limits_{|x|=1} |x'|.$ Second, we note that the right hand-side of \eqref{gaus.log.sob-w} does not depend on the choice of a homogeneous quasi-norm $|\cdot|,$ while its left-hand side does depend on it. This is, however, taken care of by the normalisation condition with $\gamma$, which also depends on $|\cdot|.$ 

\begin{proof}[Proof of Theorem \ref{semi-g-w}]
The proof is similar to that of Theorem \ref{semi-g}, so let us only indicate the differences. We keep the same relation between $f$ and $g$. The inequality 
\eqref{thm.eq.glogg} is then replaced by
\begin{eqnarray}\label{thm.eq.glogg-w}
& & \int_{\G}|x|^{-\frac{\beta(Q-2)}{Q-\beta}}|g(x)|^2\log \left(|x|^{-\frac{\beta(Q-2)}{2(Q-\beta)}}|g(x)|\right)\,d\mu \nonumber \\
& = & \int_{\G} |x|^{-\frac{\beta(Q-2)}{Q-\beta}}|f(x)|^2 \log |\gamma^{-1/2}e^{\frac{|x'|^2}{4}} 
|x|^{-\frac{\beta(Q-2)}{2(Q-\beta)}}
f(x)|\,dx \nonumber\\
& \leq & \frac{Q-\beta}{2(2-\beta)}\log \left(C \int_{\G}|\nabla_{H}f(x)|^2\,dx\right) \nonumber \\
& & +\log (\gamma^{-1/2})+\int_{\G}\frac{|x'|^2}{4} |x|^{-\frac{\beta(Q-2)}{Q-\beta}}|f(x)|^2\,dx\,,
\end{eqnarray}
where we used inequality \eqref{harsobingr-2}. We then estimate
\begin{multline}\label{EQ:aterm}
 \int_{\G}\frac{|x'|^2}{4} |x|^{-\frac{\beta(Q-2)}{Q-\beta}}|f(x)|^2\,dx=\int_{|x|\leq 1} +\int_{|x|\geq 1}  \\ \leq
 \frac{M^2}{4}\int_{|x|\leq 1} |x|^{-\frac{\beta(Q-2)}{Q-\beta}}|f(x)|^2\,dx+
 \int_{|x|\geq 1}
 \frac{|x'|^2}{4} |f(x)|^2\,dx
 \\ \leq \frac{M^2}{4}+\int_\G  \frac{|x'|^2}{4} |f(x)|^2\,dx,
\end{multline}
using that $\||x|^{-\frac{\beta(Q-2)}{2(Q-\beta)}}f\|_{L^2(\G)}=1.$
Now, combining the equality \eqref{thm.eq.nablag} together with \eqref{thm.eq.glogg-w} and \eqref{EQ:aterm}, we conclude that to show \eqref{gaus.log.sob-w}, it is enough to show that 
\begin{equation}\label{for.k-w}
\log \left(C\gamma^{-\frac{2-\beta}{Q-\beta}}e^{\left(\frac{n_1}{2}+\frac{M^2}{4} \right)\frac{2(2-\beta)}{Q-\beta}} \int_{\G}|\nabla_{H}f(x)|^2\,dx \right)\leq \frac{2(2-\beta)}{Q-\beta}\int_{\G}|\nabla_{H}f(x)|^2\,dx\,.
\end{equation}
Hence, using arguments similar to those developed in Theorem \ref{semi-g} we can check that for $\gamma$ as in \eqref{k-w} inequality \eqref{for.k-w} is satisfied, and this completes the proof of Theorem \ref{semi-g-w}. 
\end{proof}

As a curiosity, let us record a consequence of Theorem \ref{semi-g-w} for the Euclidean space $\G=\mathbb R^n$, in which case $Q=n_1=n$ and $\nabla_H=\nabla$. We also exemplify the fact that we can take any homogeneous quasi-norm on $\mathbb R^n$, so we take  a family 
$|x|_p=(|x_1|^p+\cdots+|x_n|^p)^{\frac{1}{p}}$ for $1\leq p<\infty$, and $|x|_\infty=\max\limits_{1\leq j\leq n}|x_j|.$ In this case we have \eqref{EQ:norms} with $M_p=\max\{n^{\frac12-\frac1p},1\}.$ As a result, we obtain the following weighted version of the Gross log-Sobolev inequality on $\mathbb R^n$.

\begin{cor}\label{semi-g-w-Rn}
Let $0\leq \beta <2<n$. For $1\leq p\leq\infty$, we denote by $|x|_p$ the $\ell^p$-norm of $x\in \mathbb R^n$, as above.
Let $\mu$ be the Gaussian measure on $\mathbb{R}^{n}$ given by 
$d\mu=\gamma e^{-\frac{|x|_2^2}{2}}dx$, with 
 \begin{equation}
     \label{k-w-Rn}
     \gamma:=\left(\frac{n-\beta}{2(2-\beta
     )}C e^{(n+\frac{M_p^2}{2})\frac{2-\beta}{n-\beta}-1}
     \right)^{\frac{n-\beta}{2-\beta}}\,,
 \end{equation}
 where $C$ is the constant from \eqref{harsobingr-2}, and $M_p=\max\{n^{\frac12-\frac1p},1\}.$
Then the following weighted semi-Gaussian log-Sobolev inequality is satisfied
\begin{equation}
    \label{gaus.log.sob-w-Rn}
    \int_{\mathbb R^n}|x|_p^{-\frac{\beta(n-2)}{n-\beta}}|g(x)|^2\log\left(|x|_p^{-\frac{\beta(n-2)}{2(n-\beta)}}|g(x)|\right)\,d\mu(x) \leq \int_{\mathbb R^n} |\nabla g(x)|^2\,d\mu(x)\,,
\end{equation}
for all $g$ such that $\||x|_p^{-\frac{\beta(n-2)}{2(n-\beta)}}g\|_{L^2(\mu)}=1.$
\end{cor}

\end{document}